\documentclass[12pt]{article}
 
\usepackage{myarticle}
\usepackage[colorinlistoftodos,prependcaption,textsize=tiny,backgroundcolor=black!5!white,bordercolor=red]{todonotes}

\usepackage{pgfplots}
\pgfplotsset{
  tick label style={font=\footnotesize},
  label style={font=\footnotesize},
  legend style={font=\footnotesize}
}

\newcommand{\R}{\mathbb R}             
\newcommand{\N}{\mathbb N}             
\newcommand{\eR}{\overline{\R}}        

\DeclareMathOperator*{\argmin}{argmin}

\newcommand{\iter}[1]{^{(#1)}}

\let\phitmp\phi
\let\phi\varphi
\let\varphi\phitmp
\let\epstmp\eps
\let\eps\varepsilon
\let\varepsilon\epstmp

\newcommand{\set}[1]{\lbrace #1 \rbrace}                    
\newcommand{\sint}{\mathrm{int}\,}                          
\newcommand{\ri}{\mathrm{ri}\,}                             
\newcommand{\aff}{\mathrm{aff}\,}                           

\newcommand{\seq}[2][]{(#2)_{#1}}                            
\newcommand{\norm}[2][]{\Vert #2 \Vert_{#1}}                 
\newcommand{\scal}[2]{\left\langle #1,#2 \right\rangle}      
\newcommand{\map}[3]{#1\colon #2 \to #3}                     
\newcommand{\opid}{I}                                        

\newcommand{\grad}[1][]{\nabla_{#1}}                              
\newcommand{\Hess}[1][]{\nabla^2_{#1}}                     

\newcommand{\ncone}[1]{\ensuremath{N_{#1}}}             
\renewcommand{\o}{\mathrm{o}}                           
\renewcommand{\O}{\mathcal O}                          
\newcommand{\prox}[2]{\mathrm{prox}_{#1#2}}             
\newcommand{\proj}[1]{\mathrm{proj}_{#1}}               

\newcommand{\spLin}{\mathcal L}                       
\newcommand{\spSob}[2][]{\ensuremath{H^{#2\ifthenelse{\equal{#1}{\empty}}{}{,#1}}}}

\newcommand{\setPrm}{\Omega}                            
\newcommand{\spVar}{\mathcal X}                         
\newcommand{\spPrm}{\mathcal U}                         
\newcommand{\manif}{\mathcal M}                         
\newcommand{\spPar}[1]{\mathrm{par}\,#1}                
\newcommand{\spTan}[1]{T_{#1} \manif}                   
\newcommand{\spNor}[1]{N_{#1} \manif}                   

\newcommand{\epsK}{\bm\eps ^{(K)}}
\newcommand{\epsk}{\bm\eps ^{(k)}}
\newcommand{\epsi}{\bm\eps ^{(i)}}

\newcommand{\x}{\bm x}                                  
\newcommand{\xmin}{\x^*}                                
\newcommand{\xz}{\x ^{(0)}}                             
\newcommand{\xzm}{\x ^{(-1)}}                           
\newcommand{\xk}{\x ^{(k)}}                             
\newcommand{\xkp}{\x ^{(k+1)}}                          
\newcommand{\xkm}{\x ^{(k-1)}}                          
\newcommand{\xK}{\x ^{(K)}}                             

\newcommand{\z}{\bm z}                                  
\newcommand{\zk}{\z ^{(k)}}                             

\newcommand{\sslow}{\underline\alpha}
\newcommand{\ssup}{\bar\alpha}

\renewcommand{\u}{\bm u}                                
\newcommand{\ud}{\dot\u}                                

\newcommand{\givenPrm}{\u^*}

\newcommand{\y}{\bm y}                                  
\newcommand{\yk}{\y ^{(k)}}                             

\newcommand{\w}{\bm w}                                  
\newcommand{\wk}{\w ^{(k)}}                             

\renewcommand{\b}{\bm b}                                
\newcommand{\bk}{\b ^{(k)}}                             

\renewcommand{\v}{\bm v}                                

\newcommand{\bmu}{\bm\mu}
\newcommand{\bmuk}{\bmu^{(k)}}

\newcommand{\e}{\bm e}                                  
\newcommand{\ez}{\e ^{(0)}}                             
\newcommand{\ek}{\e ^{(k)}}                             
\newcommand{\ekp}{\e ^{(k+1)}}                          
\newcommand{\eK}{\e ^{(K)}}                             

\renewcommand{\d}{\bm d}                                  
\newcommand{\dk}{\d ^{(k)}}                             

\newcommand{\argmap}{\psi}

\newcommand{\wein}[3][]{\mathfrak W_{#1} (#2, #3)}
\newcommand{\bwd}[1][]{\mathrm{P}_{\alpha_{#1}g}}                  
\newcommand{\pgd}[1][]{\mathcal{A}_{\alpha_{#1}}}
\newcommand{\apg}[1][]{\mathcal{A}_{\alpha_{#1}, \beta_{#1}}}

\newcommand{\projTan}{\Pi}
\newcommand{\projNor}{\Pi^\perp}
\newcommand{\rstDom}[2]{\left.{#1}\right|_{#2}}                     

\newcommand{\fixMap}[1][]{\mathcal A_{#1}}
\newcommand{\neighbourhood}[1][]{\mathcal N_{#1}}
\newcommand{\neighbourhoodPRM}{U}
\newcommand{\neighbourhoodFM}{V}

\newcommand{\neighbourhoodAPG}[1][]{V_{\alpha_{#1}}}

\newcommand{\PGD}{\mathrm{PGD}}
\newcommand{\APG}{\mathrm{APG}}

\newcommand{\deriv}[1][]{D_{#1}}
\newcommand{\derivx}{\deriv[\x]}
\newcommand{\derivz}{\deriv[\z]}
\newcommand{\derivu}{\deriv[\u]}

\newcommand{\seeAppSec}[2]{#1 Section~\ref{#2} in the appendix}
\newcommand{\appSecLabel}[1]{ssec:#1:proof}
\newcommand{\findProofIn}[1]{\begin{proof}The proof is in Section~\ref{\appSecLabel{#1}} in the appendix.\end{proof}}
\newcommand{\appSubSect}[2]{\subsection{Proof of {#1}~\ref{#2}} \label{\appSecLabel{#2}}}

\newcommand{\crefItms}[2]{%
  \hyperref[#2]{\namecref{#1}~\labelcref*{#1}\ref*{#2}}%
}

\graphicspath{{./figures/}}

\KEYWORDS{Partial Smoothness, Automatic Differentiation, Fixed-Point Algorithm, Bilevel Learning}
\begin{document}

\thispagestyle{empty}
\begin{center}
\vspace*{0.03\paperheight}
{\Large\bf Automatic Differentiation of Optimization \\[0.8ex] Algorithms with Time-Varying Updates \par}
\bigskip
\bigskip
{\large
Sheheryar Mehmood$^\star$ and Peter Ochs$^\dagger$ \\ \medskip
{\small
$^\star$~University of T\"{u}bingen, T\"{u}bingen, Germany \\
$^\dagger$~Saarland University, Saarbr\"{u}cken, Germany \\
}
}
\end{center}
\bigskip

\begin{abstract}
    Numerous Optimization Algorithms have a time-varying update rule thanks to, for instance, a changing step size, momentum parameter or, Hessian approximation. In this paper, we apply unrolled or automatic differentiation to a time-varying iterative process and provide convergence (rate) guarantees for the resulting derivative iterates. We adapt these convergence results and apply them to proximal gradient descent with variable step size and FISTA when solving partly smooth problems. We confirm our findings numerically by solving $\ell_1$ and $\ell_2$-regularized linear and logisitc regression respectively. Our theoretical and numerical results show that the convergence rate of the algorithm is reflected in its derivative iterates.
\end{abstract}



\section{Introduction} \label{sec:intro}

For some parameter $\u\in\spPrm$, we consider the following parametric iterative process
\begin{equation} \tag{$\mathcal{R}$} \label{itr:FPI}
    \xkp (\u) \coloneqq \fixMap (\xk (\u), \u) \,,
\end{equation}
for $k\geq0$, where $\xz\in\spVar$ is the initial iterate and $\map{\fixMap}{\spVar\times\spPrm}{\spVar}$ is the update mapping. The iterates $\xk (\u)$ generated by \eqref{itr:FPI} depend on $\u$ due to the dependence of $\fixMap$ on $\u$. The goal of performing the iterations \eqref{itr:FPI} is mainly to solve the non-linear equation
\begin{equation} \tag{$\mathcal{R}_\mathrm{e}$} \label{eq:FPE}
    \x = \fixMap (\x, \u) \,,
\end{equation}
with respect to $\x$ for each parameter $\u$. A simple example is gradient descent with appropriate step size $\alpha>0$ where we define $\fixMap (\x, \u) \coloneqq \x - \alpha \grad[\x] F (\x, \u)$ for solving a parametric optimization problem of the form
\begin{equation} \tag{$\mathcal{P}$} \label{prob:min}
    \min_{\x\in\spVar} \ F (\x, \u) \,,
\end{equation}
with a smooth objective $\map{F}{\spVar\times\spPrm}{\R}$. In this case \eqref{eq:FPE} reduces to the optimality condition $\grad[\x] F (\x, \u) = 0$.

Often in practice, for some smooth $\map{\ell}{\spVar\times\spPrm}{\R}$, we encounter problems of the form
\begin{equation} \label{prob:min:upper}
    \min_{\u\in\spPrm} \ \ell (\argmap (\u), \u) \,.
\end{equation}
where, given any $\u\in\spPrm$, $\argmap (\u)$ is a solution of \eqref{eq:FPE} or \eqref{prob:min}. This kind of optimization, formally known as bilevel optimization \cite{DKP+15, DZ20} has emerged as a crucial tool in machine learning, playing a foundational role in various applications such as hyperparameter optimization \cite{Ben00, Dom12, MDA15a}, implicit neural networks \cite{AK17, AAB+19, BKK19, BKK20, WK20}, meta-learning \cite{HAM+21}, and neural architecture search \cite{LSY19}.

Solving \eqref{prob:min:upper} through gradient-based methods requires computing the gradient of the upper objective $\ell (\argmap (\u), \u)$ with respect to $\u$ which, in turn, requires computing the derivative of the solution mapping $\u\mapsto\argmap (\u)$, that is, $\deriv \argmap (\u)$ by using the chain rule. When $\ell$ and $F$ (or $\fixMap$) satisfy certain smoothness and regularity assumptions, this can be achieved by two different methods. (i) One is known as \textit{Implicit Differentiation} (ID) which leverages the classical Implicit Function Theorem or IFT
\cite[Theorem~1B.1]{DR09} applied to \eqref{eq:FPE} to estimate
\begin{equation} \label{eq:FPE:IFT:intro}
    \begin{aligned}
        \deriv \argmap (\u) &= \left(\opid - \derivx \fixMap (\x, \u)\right)^{-1} \derivx \fixMap (\x, \u) \\
        &= -\Hess[\x] F (\x, \u)^{-1} \derivu \grad[x] F (\x, \u) \,,
    \end{aligned}
\end{equation}
where $\x\coloneqq\argmap (\u)$ (see Theorem~\ref{thm:IFT:FPE} for more details). The key assumption for IFT is the strict bound on the spectral radius $\rho (\derivx \fixMap (\argmap (\u), \u)) < 1$, which we call the contraction property since it is a sufficient condition for $\fixMap$ to be $1$-Lipschitz or contractive. A direct consequence is the linear convergence of $\xk (\u)$ to $\argmap (\u)$ \cite[Section~2.1.2,~Theorem~1]{Pol87}.
The output of ID \eqref{eq:FPE:IFT:intro} depends on accurately solving \eqref{eq:FPE} --- possibly through \eqref{itr:FPI} --- and inverting the linear system efficiently. (ii) The other technique is known as \textit{Unrolled, Iterative or Automatic Differentiation} (AD) or simply \textit{Unrolling} where $\deriv \argmap (\u)$ is estimated by applying automatic differentiation \cite{Wen64, Lin70, GW08} to the output $\xK (\u)$ of \eqref{itr:FPI} after $K\in\N$ iterations. The update rule for forward mode AD applied to $\xK$ is given by
\begin{equation} \tag{$\mathcal{DR}$} \label{itr:D:FPI}
    \deriv \xkp (\u) \coloneqq \derivx \fixMap (\xk (\u), \u) \deriv \xk (\u) + \derivu \fixMap (\xk (\u), \u) \,,
\end{equation}
for $0\leq k< K$ where we may assume $\deriv \xz(\u) = 0$. The effective use of AD relies on the guarantee that $\deriv \xK (\u)$ accurately estimates $\deriv\argmap (\u)$ for a preferably small $K$. More formally, we require the convergence of $\deriv \xK (\u)$ to $\deriv\argmap (\u)$ as $K\to\infty$, with a convergence rate comparable to that of $\xk (\u)$. The reverse mode AD or Backpropagation \cite{RHW86, BPR+18} has a different update rule, however, the result computed is the same in the end \cite{GW08}. Due to the nature of the iterations \eqref{itr:D:FPI}, we will use $\xK$ and $\xk$ interchangeably.

The reverse mode AD has a memory overhead since it requires storing the iterates $\seq[0\leq k < K]{\xk}$. Even though this is in contrast with ID which only depends on $\xK$, AD is still a popular choice for estimating $\deriv \argmap (\u)$. A crucial advantage of AD is that it provides a nice blackbox implementation thanks to the powerful autograd libraries included in PyTorch \cite{PGM+19}, TensorFlow \cite{ABC+16}, and JAX \cite{BFH+18} whereas ID requires either a custom implementation \cite[see~Remark~2]{BPV23} or using special libraries \cite[etc.]{BBC+22, RFL+22}.

\subsection{Convergence of Automatic Differentiation} \label{ssec:intro:RW}

The convergence of the derivative sequence $\deriv \xk (\u)$ for AD of \eqref{itr:FPI} were studied classically for linear \cite{Fis91} and general iteration maps \cite{Gil92}. More recently, the convergence analysis was provided for gradient descent \cite{LVD20, MO20}, the
Heavy-ball method \cite{MO20}, proximal gradient descent for lasso-like problems \cite{BKM+22}, the Sinkhorn-Knopp algorithm \cite{PV23}, and super-linearly convergent algorithms \cite{BPV23}. The convergence results were extended to non-smooth iterative processes in \cite{BPV22} using conservative calculus \cite{BP20, BP21}. Ablin et al.~\cite{APM20} theoretically established the better convergence rate for computing the gradient of the value function through AD and ID. Usually, the proof of convergence of $\deriv \xk (\u)$ relies on the contraction property in each of the above works with an exception of \cite{PV23}. The linear rate of convergence requires (local) Lipschitz continuity of $\deriv \fixMap$. A slightly different line of study on the problematic behaviour of automatic differentiation in the earlier iterations was carried out in \cite{SGB+22}.

\subsection{Iteration-Dependent Updates} \label{ssec:intro:FPI:k}

Most iterative algorithms for solving \eqref{prob:min} which include gradient descent with line search, Nesterov's accelerated gradient \cite{Nes83}, and Quasi-Newton methods \cite[etc.]{Dav59, FP63, Bro70, Fle70} cannot be modelled by \eqref{itr:FPI}. This happens because the update mapping is changing at every iteration; $\fixMap$ in \eqref{itr:FPI} is replaced by $\map{\fixMap[k]}{\spVar\times\spPrm}{\spVar}$ giving us the update
\begin{equation} \tag{$\mathcal{R}_k$} \label{itr:FPI:k}
    \xkp (\u) \coloneqq \fixMap[k] (\xk (\u), \u) \,.
\end{equation}
For example, for gradient descent with line search, $\fixMap[k] (\x, \u) \coloneqq \x - \alpha_k \grad[\x] F (\x, \u)$. When $\fixMap[k]$ is $C^1$-smooth for all $k\in\N$, the modified update rule for performing forward mode AD on \eqref{itr:FPI:k} reads:
\begin{equation} \tag{$\mathcal{DR}_k$} \label{itr:D:FPI:k}
    \deriv \xkp (\u) \coloneqq \derivx \fixMap[k]  (\xk (\u), \u) \deriv \xk (\u) + \derivu \fixMap[k] (\xk (\u), \u) \,.
\end{equation}

Beck~\cite{Bec94} studied this problem for the $C^1$-smooth sequence of functions $\fixMap[k]$ with uniformly bounded derivatives and $\deriv \fixMap[k]$ converging to $\deriv \fixMap$ pointwise. They showed that for such a sequence, when $\seq[k\in\N]{\xk (\u)}$ is generated by \eqref{itr:FPI:k} to solve the system \eqref{eq:FPE}, then $\seq[k\in\N]{\deriv\xk (\u)}$ generated by \eqref{itr:D:FPI:k} converges to $\deriv \argmap (\u)$ defined in \eqref{eq:FPE:IFT:intro}. \textit{However, their results cannot account for the linear convergence of $\deriv\xk (\u)$ in their given setting (see Section~\ref{ssec:pre:FPI:k:Beck:supp}). Moreover, their setting excludes the more general cases where $\fixMap[k]$ does not converge, for instance, gradient descent with line search.}

Griewank et al.~\cite{GBC+93} studied the special setting for solving the non-linear system $G (\x, \u) = 0$ with respect to $\x$ through preconditioned Picard iterations, that is, $\fixMap[k] (\x, \u) \coloneqq \x - P_k (\u) G (\x, \u)$. The parameter-dependent preconditioner $\map{P_k}{\spPrm}{\spLin (\spVar, \spVar)}$ is changing at every iteration and may not converge. They demonstrated linear convergence of $\deriv\xk (\u)$ for a $C^1$-smooth $\map{G}{\spVar\times\spPrm}{\spVar}$ with bounded maps $\derivx G$ and $(\x, \u)\mapsto \derivx G (\x, \u)^{-1}$ and Lipschitz continuous map $\deriv G$. The map $P_k$ is assumed to be $C^1$-smooth with bounded (potentially $0$) derivative for every $k\in\N$.

Recently in \cite{Rii20, MO24}, AD was applied to \eqref{itr:FPI:k} when $\fixMap[k]$ denotes the update mapping of FISTA \cite{BT09} (which we will call Accelerated Proximal Gradient or APG) applied to \eqref{prob:min} for non-smooth $F \coloneqq f + g$ with step size $\alpha_k$ and extrapolation parameter $\beta_k$. They adapted the results of \cite{Bec94} and demonstrated the convergence of $\deriv \xk (\u)$ to $D\argmap (\u)$ when $f$ is $C^2$-smooth with a Lipschitz continuous gradient, $g$ is partly smooth \cite{Lew02}, the sequences $\seq[k\in\N]{\alpha_k}$ and $\seq[k\in\N]{\beta_k}$ converge and $F$ additionally satisfies restricted injectivity and non-degeneracy conditions \cite{Lew02}. $\deriv \argmap (\u)$ is computed by applying the Implicit Function Theorem for partly smooth functions \cite[Theorem~5.7]{Lew02}. Partly smooth functions are a class of special non-smooth functions which exhibit smoothness when restricted to a low-dimensional manifold $\manif$ and model various practical regularization functions \cite{VDP+17}. The non-smooth or sharp behaviour occurs only when we move orthogonal to $\manif$. The combination of restricted injectivity and non-degeneracy assumptions induces the contraction property in proximal gradient descent or PGD \cite{LFP14} and APG \cite{LFP17} near the solution. \textit{The results of \cite{Rii20, MO24} do not explain the linear convergence of AD of APG or even the convergence of AD of PGD with non-converging variable step size.}

\subsection{Contributions} \label{ssec:intro:cont}

In this paper, we aim to resolve the theoretical gaps presented in Section~\ref{ssec:intro:FPI:k} and therefore reinforce the usefulness of AD. In particular, the main contributions of this paper are the following:
\begin{enumerate}[label=(\roman*)]
    \item \label{itm:cont:Beck:rates} We strengthen the results of \cite{Bec94} by providing a convergence rate analysis for $\deriv \xk (\u)$ generated by \eqref{itr:D:FPI:k} after equipping their setting with an additional assumption that is satisfied by most optimization algorithms (see Assumption~\ref{ass:basic:modified}\ref{itm:basic:modified:fixed:point} and Remark~\ref{rem:basic:modified}).
    \item \label{itm:cont:NoLim:conv:rates} We establish (linear) convergence of $\deriv \xk (\u)$ beyond the setting in \ref{itm:cont:Beck:rates} without the pointwise convergence condition of $\fixMap[k]$ to $\fixMap$ removed (see Assumption~\ref{ass:basic:modified}\ref{itm:basic:modified:fixed:point} and \ref{ass:basic:modified}\ref{itm:basic:modified:spectral:norm}).
    \item \label{itm:cont:PGD:APG:conv:rates} We demonstrate the convergence of $\deriv\xk (\u)$ for Proximal Gradient Descent \cite{LM79} or PGD with variable step size and Accelerated Proximal Gradient \cite{BT09} or APG by adapting the results for \ref{itm:cont:NoLim:conv:rates} and \ref{itm:cont:Beck:rates} respectively. We show that the rate of convergence of $\xk (\u)$ is reflected in that of $\deriv \xk (\u)$ for both algorithms.
\end{enumerate}

\subsection{Bilevel Optimization beyond Automatic and Implicit Differentiation}
It is worthwhile to point out a different method for solving \eqref{prob:min:upper} in addition to AD and ID. The idea is to rewrite the problem by incorporating the value function through an inequality constraint \cite{DMN10, LYW+22, SC23}. This reduces the nested problem into a single problem which is then solved through penalty-based methods. This technique is sometimes referred to as single-loop or Hessian-free approach and has been shown to give promising results \cite{KKW+23}.

\subsection{Notation}
In this paper, $\spVar$ and $\spPrm$ are Euclidean spaces and $\eR\coloneqq[-\infty, \infty]$. With a slight abuse of notation, $\norm{\cdot}$ denotes any norm on these spaces including the one induced by the inner product $\scal{\cdot}{\cdot}$. Any operator norm on the space of linear operators between $\spVar$ and $\spPrm$ is also denoted by $\norm{\cdot}$. We say that $\wk = \o (\zk)$ iff $\norm{\wk} = \o (\norm{\zk})$ and $\wk = \O (\zk)$ iff $\norm{\wk} = \O (\norm{\zk})$. 
We use standard notions of Riemannian Geometry and partly smooth functions found in \cite{Lew02,Lia16,VDP+17}. A $C^2$-smooth embedded submanifold $\manif\subset\spVar$ of $\spVar$ is simply referred to as a manifold. For any $\x\in\manif$, $\projTan (\x) \coloneqq \proj{\spTan{\x}}$ and $\projNor (\x) \coloneqq \proj{\spNor{\x}}$ denote the orthogonal projection onto $\spTan{\x}$ and $\spNor{\x}$, the tangent and normal spaces respectively and for any $\v\in\spTan{\x}$, $\wein[\x]{\cdot}{\v}$ represents the Weingarten map. For any $\alpha>0$ and proper and lower semi-continuous $\map{g}{\spVar\times U}{\eR}$ where $g(\cdot, \u)$ is convex for all $\u\in U \subset \spPrm$, we define $\map{\bwd}{\spVar\times U}{\spVar}$ by $\bwd (\cdot, \u) = \prox{\alpha}{g(\cdot, \u)}$, that is,
\begin{equation} \label{eq:prox}
    \bwd (\w, \u) \coloneqq \argmin_{\x\in\spVar} \alpha g(\x, \u) + \frac{1}{2} \norm{\x - \w}^2 \,.
\end{equation}
We use the terminology of \cite{RW98} for the subdifferential of $g$. When $U$ is open and for any $\x \in\spVar$, $g(\x, \cdot)$ is differentiable at $\u\in U$, we obtain $\partial g (\x, \u) = \partial_{\x} g (\x, \u) \times \set{\grad[\u] g (\x, \u)} $ \cite[Lemma~1]{MO24}.
When $\rstDom{g}{\manif\times U}$ is $C^2$-smooth, we follow \cite[Remark~22(iv)]{MO24} to denote its Riemannian gradient and Hessian.

\section{Differentiation of Iteration-Dependent Algorithms} \label{sec:FPI:k}
In this section, we lay down the assumptions on the sequence $\seq[k\in\N]{\fixMap[k]}$ and provide convergence and convergence rate guarantees for the derivative iterations defined in \eqref{itr:D:FPI:k}. In particular we establish (i) convergence rates when $\fixMap[k]$ has a pointwise limit $\fixMap$, and (ii) convergence and convergence rates when $\fixMap[k]$ does not have a pointwise limit. For a better comparison and understanding, this section is adapted to follow the same pattern as Section~\ref{ssec:pre:FPI:k:Beck:supp}, which summarizes the results by \cite{Bec94}.

\subsection{Problem Setting} \label{ssec:FPI:k:PS}
For a given $\givenPrm$, we assume that $\xmin$ is the solution that we are trying to estimate through \eqref{itr:FPI:k}. Because the pointwise limit $\fixMap$ of $\fixMap[k]$ may not exist, we assume that $\xmin$ is a fixed point of $\fixMap[k](\cdot, \givenPrm$) for every $k\in\N$. This seemingly restrictive assumption is satisfied by the optimization algorithms highlighted in Section~\ref{ssec:intro:FPI:k}. The application of the Implicit Function Theorem on this sequence of fixed-point equations requires the contraction property, that is, the strict bound on some operator norm $\norm{\derivx \fixMap[k] (\xmin, \givenPrm)} < 1$ eventually for all $k\in\N$. Furthermore, we assume that the sequence of affine maps $\seq[k\in\N]{X \mapsto \derivx\fixMap[k](\xmin, \givenPrm) X + \derivu\fixMap[k](\xmin, \givenPrm)}$ share a common fixed-point $X_*\in\spLin (\spVar, \spVar)$ --- the reason will become clear shortly. Finally, to show the convergence of the derivative sequence $\seq[k\in\N]{\deriv\xk (\givenPrm)}$, we assume that the sequence $\seq[k\in\N]{\xk (\givenPrm)}$ converges to $\xmin$. More formally, given $(\xz, \xmin, \givenPrm, X_*) \in \spVar\times\spVar\times\spPrm\times\spLin (\spVar, \spVar)$, $\neighbourhoodFM \in \neighbourhood[(\xmin, \givenPrm)]$, $\seq[k\in\N_0]{\fixMap[k]}$, and a norm $\norm{\cdot}$ on $\spLin (\spVar, \spVar)$ induced by some vector norm, we assume that the following conditions hold.
\begin{assumption} \label{ass:basic:modified}
    \begin{enumerate}[label=(\roman*)]
        \item \label{itm:basic:modified:C1} $\rstDom{\fixMap[k]}{\neighbourhoodFM}$ is $C^1$-smooth for all $k$,
        \item \label{itm:basic:modified:fixed:point} $\xmin = \fixMap[k] (\xmin, \givenPrm)$ and $X_* = \derivx \fixMap[k] (\xmin, \givenPrm) X_* + \derivu \fixMap[k] (\xmin, \givenPrm) $ for all $k$,
        \item \label{itm:basic:modified:spectral:norm} $\limsup_{k\to\infty} \norm{\derivx \fixMap[k] (\xmin, \givenPrm)} < 1$, and
        \item \label{itm:basic:modified:iterates} $\xk (\givenPrm)$ generated by \eqref{itr:FPI:k} has limit $\xmin$ such that $(\xk (\givenPrm), \givenPrm) \in \neighbourhoodFM$ for all $k$.
    \end{enumerate}
\end{assumption}
\begin{remark} \label{rem:basic:modified}
    A special case of Assumption~\ref{ass:basic:modified}\ref{itm:basic:modified:spectral:norm} arises when for some $C^1$-smooth mapping $\map{\fixMap}{\neighbourhoodFM}{\spVar}$, we have $\xmin = \fixMap (\xmin, \givenPrm)$, $\deriv \fixMap[k] (\xmin, \givenPrm) \to \deriv \fixMap (\xmin, \givenPrm)$ and $\rho (\derivx \fixMap (\xmin, \givenPrm)) < 1$.
\end{remark}

\subsection{Implicit Differentiation}
Combining Assumptions~\ref{ass:basic:modified}(\ref{itm:basic:modified:C1}--\ref{itm:basic:modified:spectral:norm}) allows us to invoke the Implicit Function Theorem on $\xmin = \fixMap[k] (\xmin, \givenPrm)$ for every $k$ large enough and obtain a $C^1$-smooth fixed-point map $\argmap_{k}$ on a neighbourhood $\neighbourhoodPRM_{k} \in \neighbourhood[\givenPrm]$. In particular, we have the following result.
\begin{lemma} \label{lem:IFT:FPE:k}
    Let $(\xmin, \givenPrm, X_*) \in \spVar\times\spPrm\times\spLin (\spVar, \spVar)$, $\neighbourhoodFM \in \neighbourhood[(\xmin, \givenPrm)]$ $\seq[k\in\N_0]{\fixMap[k]}$, and $\norm{\cdot}$ be such that Assumption~\ref{ass:basic:modified} (\ref{itm:basic:modified:C1}--\ref{itm:basic:modified:spectral:norm}) is satisfied. Then there exists $K\in\N$ such that $\forall k\geq K$, there exists $\neighbourhoodPRM_k\in\neighbourhood[\givenPrm]$ and a $C^1$-smooth $\map{\argmap_k}{\neighbourhoodPRM_k}{\spVar}$ such that $\forall \u\in \neighbourhoodPRM_k$, $\argmap_k(\u) = \fixMap[k](\argmap_k(\u), \u)$, and
    \begin{equation} \label{eq:FPE:IFT:k}
        \deriv\argmap_k(\u) = \big( \opid - \derivx \fixMap[k] (\argmap_k(\u), \u) \big)^{-1} \derivu \fixMap[k] (\argmap_k(\u), \u) \,,
    \end{equation}
    so that $\deriv \argmap_{k} (\givenPrm) = X_*$.
\end{lemma}
\findProofIn{lem:IFT:FPE:k}
\begin{remark} \label{rem:IFT:FPE:k}
    If we replace Assumptions~\ref{ass:basic:modified}\ref{itm:basic:modified:fixed:point} and \ref{ass:basic:modified}\ref{itm:basic:modified:spectral:norm} by a slightly stronger assumption, that is, ``\textit{there exists $\neighbourhoodPRM\in\neighbourhood[\givenPrm]$ such that for any $(\x, \u)\in \spVar\times\neighbourhoodPRM$, $\x = \fixMap[k] (\x, \u)$ for some $k\in\N$ implies $\x = \fixMap[k] (\x, \u)$ for every $k\in\N$, and $\limsup_{k\to\infty} \norm{\derivx \fixMap[k] (\x, \u)} < 1$}'', we obtain a shared $C^1$-smooth fixed-point map $\map{\argmap}{\neighbourhoodPRM}{\spVar}$, such that, $\neighbourhoodPRM_k = \neighbourhoodPRM$ and $\argmap_k = \argmap$ for all $k\in\N$.
\end{remark}

\subsection{Automatic Differentiation}
In Lemma~\ref{lem:IFT:FPE:k}, we just showed that under Assumption~\ref{ass:basic:modified}, the common fixed-point $X_*$ of the mappings $X \mapsto \derivx\fixMap[k](\xmin, \givenPrm) X + \derivu\fixMap[k](\xmin, \givenPrm)$ for $k\in\N$ represents the derivative of the solution at $\u = \givenPrm$. We now show that the derivative sequence generated by \eqref{itr:D:FPI:k} for $\u = \givenPrm$ converges to $X_*$.
Instead of assuming a pointwise limit for $\deriv \fixMap[k]$, we assert that $\deriv \fixMap[k] (\xk (\givenPrm), \givenPrm) - \deriv \fixMap[k] (\xmin, \givenPrm)$ tends to $0$.
\begin{assumption} \label{ass:convergence:modified}
    The sequence $\seq[k\in\N_0]{\deriv\fixMap[k] (\xk (\givenPrm), \givenPrm) - \deriv\fixMap[k] (\xmin, \givenPrm)}$ converges to $0$.
\end{assumption}
\begin{remark} \label{rem:convergence:modified}
    \begin{enumerate}[label=(\roman*)]
        \item A sufficient condition for Assumption~\ref{ass:convergence:modified} is
        equicontinuity of $\seq[k\in\N_0]{\fixMap[k]}$ at $(\xmin, \givenPrm)$.
        \item Combining Assumption~\ref{ass:convergence:modified} with the special case of Remark~\ref{rem:basic:modified}, we obtain the setting of \cite{Bec94}, that is, $\deriv\fixMap[k] (\xk (\givenPrm), \givenPrm) - \deriv\fixMap (\xmin, \givenPrm) \to 0$ (\seeAppSec{see also}{ssec:pre:FPI:k:Beck:supp}).
    \end{enumerate}
\end{remark}
We are now ready to establish the convergence of the derivative sequence $\seq[k\in\N]{\deriv\xk (\givenPrm)}$ to $X_*$ when Assumptions~\ref{ass:basic:modified} and \ref{ass:convergence:modified} hold.
\begin{theorem} \label{thm:D:FPI:modified:conv}
    Let $(\xz, \xmin, \givenPrm) \in \spVar\times\spVar\times\spPrm$, $\neighbourhoodFM \in \neighbourhood[(\xmin, \givenPrm)]$, and $\seq[k\in\N_0]{\fixMap[k]}$ be such that Assumptions~\ref{ass:basic:modified} and \ref{ass:convergence:modified} are satisfied. Then the sequence $\seq[k\in\N_0]{\deriv\xk (\givenPrm)}$ generated by \eqref{itr:D:FPI:k} converges to $X_*$ which is equal to $\left(\opid - \derivx \fixMap[k] (\xmin, \givenPrm) \right)^{-1} \derivu \fixMap[k] (\xmin, \givenPrm)$ for all $k\in\N$ (see Assumption~\ref{ass:basic:modified}\ref{itm:basic:modified:fixed:point}).
\end{theorem}
\findProofIn{thm:D:FPI:modified:conv}
The derivative iterates $\deriv \xk (\givenPrm)$ additionally converge with a linear rate when the sequences $\xk (\givenPrm)$ and $\deriv \fixMap[k] (\xk (\givenPrm), \givenPrm) - \deriv \fixMap[k] (\xmin, \givenPrm)$ converge linearly.
\begin{assumption} \label{ass:linear:convergence:modified}
    The sequence $\seq[k\in\N_0]{\xk (\givenPrm)}$ converges linearly to $\xmin$ and
    \begin{equation} \label{ineq:Lipschitz:Like:modified}
        \deriv\fixMap[k] (\xk (\givenPrm), \givenPrm) - \deriv\fixMap[k] (\xmin, \givenPrm) = \O (\xk (\givenPrm) - \xmin) \,.
    \end{equation}
\end{assumption}
\begin{remark}
    A sufficient condition for Assumption~\ref{ass:linear:convergence:modified} is the Lipschitz continuity of $\seq[k\in\N_0]{\deriv\fixMap[k]}$ in $\x$ uniformly in $\u$ and $k$.
\end{remark}
\begin{theorem} \label{thm:D:FPI:modified:conv:linear}
    Let $(\xz, \xmin, \givenPrm) \in \spVar\times\spVar\times\spPrm$, $\neighbourhoodFM \in \neighbourhood[(\xmin, \givenPrm)]$, and $\seq[k\in\N_0]{\fixMap[k]}$ be such that Assumptions~\ref{ass:basic:modified} and \ref{ass:linear:convergence:modified} are satisfied. Then the sequence $\seq[k\in\N_0]{\deriv\xk (\givenPrm)}$ generated by \eqref{itr:D:FPI:k} converges linearly to $X_*$. In particular, for all $\delta \in (0, 1 - \rho)$, there exist $C_1(\delta)$, $C_2(\delta)$ and $K\in\N$, such that for all $k\geq K$, we have
    \begin{equation} \label{eq:D:FPI:modified:conv_rates}
        \begin{aligned}
            \norm{\deriv \xk (\givenPrm) - X_*} &\leq C_1 (\delta) (k-K) q^{k-K} \\ &+ C_2 (\delta) (\rho + \delta)^{k-K} \,,
        \end{aligned}
    \end{equation}
    where $\rho \coloneqq \rho (\derivx \fixMap (\xmin, \givenPrm))$, $q \coloneqq \max(\rho + \delta, q_{\x})$, and $q_{\x} < 1$ is the convergence rate of the sequence $\seq[k\in\N_0]{\xk}$.
\end{theorem}
\findProofIn{thm:D:FPI:modified:conv:linear}
Owing to Corollary~\ref{cor:LGFPI:modified:conv} and our discussion in Remark~\ref{rem:basic:modified}, the following result shows that the work of \cite{Bec94} is reproduced and strengthened with a convergence rate guarantee under the additional requirement of Assumption~\ref{ass:basic:modified}\ref{itm:basic:modified:fixed:point}.
\begin{corollary} \label{cor:D:FPI:modified:conv:(linear)}
    The conclusions of Theorems~\ref{thm:D:FPI:modified:conv} and \ref{thm:D:FPI:modified:conv:linear} also hold when, in their respective hypotheses, Assumption~\ref{ass:basic:modified}\ref{itm:basic:modified:spectral:norm} is replaced by ``there exists a $C^1$-smooth $\map{\fixMap}{\neighbourhoodFM}{\spVar}$, such that $\deriv \fixMap[k] (\xmin, \givenPrm) $ converges to $ \deriv \fixMap (\xmin, \givenPrm)$ and $\rho (\derivx \fixMap (\xmin, \givenPrm)) < 1$''.
\end{corollary}
\begin{remark} \label{rem:D:FPI:modified:conv:(linear)}
    The conclusions of Theorems~\ref{thm:D:FPI:modified:conv} and \ref{thm:D:FPI:modified:conv:linear} and Corollary~\ref{cor:D:FPI:modified:conv:(linear)} do not require the neighbourhood $\neighbourhoodFM$ of $(\xmin, \givenPrm)$ to be a connected set as long as Assumptions~\ref{ass:convergence:modified} and \ref{ass:linear:convergence:modified} are respectively satisfied along with Assumption~\ref{ass:basic:modified}\ref{itm:basic:modified:iterates}.
\end{remark}

\section{Differentiation of Proximal Gradient-type Methods} \label{sec:FBS}

We now turn our attention to unrolling PGD with variable step size and APG when solving 
\begin{equation} \label{prob:min:comp}
    \min_{\x\in\spVar} \ F (\x, \u) \,, \quad F \coloneqq f + g \,,
\end{equation}
and aim at providing convergence (rate) guarantees for the corresponding derivative sequences. The idea is to consider the class of partly smooth functions \cite{Lew02} which models a wide range of non-smooth regularizers and constraints appearing in practice \cite{VDP+17}. Furthermore, this setting (see Assumption~\ref{ass:CPSO}) under some additional assumptions (see Assumptions~\ref{ass:RPD} and \ref{ass:ND}) yields a more general IFT (see Theorem~\ref{thm:IFT}) for differentiating the solution mapping of \eqref{prob:min:comp}. This in turn allows us to differentiate the update mappings of PGD (see Theorem~\ref{thm:PGD:Derv}) and APG (see Theorem~\ref{thm:APG:Derv}), since they are also the solution mappings of partly smooth functions (see Equations~\ref{eq:prox}, \ref{eq:FixMap:PGD} and \ref{eq:FixMap:APG}).

\subsection{Problem Setting}

We first define our setting in a more precise manner. Let $\manif$ be a $C^2$-smooth manifold and $\setPrm\subset\spPrm$ be an open set, we need $f$ and $g$ to satisfy the following assumption.
\begin{assumption}[Convex Partly Smooth Objective] \label{ass:CPSO}
$\map{f}{\spVar\times\spPrm}{\R}$ is $C^2$-smooth, $\map{g}{\spVar\times\spPrm}{\eR}$ is partly smooth relative to $\manif\times\setPrm$ and for every $\u\in\setPrm$, $f(\cdot, \u)$ and $g(\cdot, \u)$ are convex, $f(\cdot, \u)$ has an $L$-Lipschitz continuous gradient, and $g (\cdot, \u)$ has a simple proximal mapping.
\end{assumption}
Let $\xmin\in\manif$ be a solution of \eqref{prob:min:comp} at $\givenPrm\in\setPrm$. When $g = 0$, IFT requires $\grad[\x] f (\xmin, \givenPrm) \succ 0$ which is also a sufficient condition for the linear convergence of gradient descent and the Heavy-ball method \cite{Pol87}. When $g$ is partly smooth, IFT requires positive definiteness of $\Hess[\manif] F (\xmin, \givenPrm)$ \cite{Lew02} while the linear convergence of PGD and APG is established by the positive definiteness of $\projTan (\xmin) \Hess[\x] f (\xmin, \givenPrm) \projTan (\xmin)$ \cite{LFP14, LFP17}.
\begin{assumption}[Restricted Positive Definiteness] \label{ass:RPD} Let $(\xmin, \givenPrm) \in \manif\times\setPrm$ be a given point.
\begin{enumerate}[label=(\roman*)]
    \item \label{itm:RPD-i}The Hessian $\Hess[\manif] F (\xmin, \givenPrm)$ is positive definite on $\spTan\xmin $, that is,
    \begin{equation} \tag{RPD-i} \label{eq:RPD-i}
        \Hess[\manif] F (\xmin, \givenPrm) \succ 0 \,.
    \end{equation}
    \item \label{itm:RPD-ii}Moreover, $\projTan (\xmin) \Hess[\x] f (\xmin, \givenPrm)$ is positive definite on $\spTan{\xmin}$, that is,
    \begin{equation} \tag{RPD-ii} \label{eq:RPD-ii}
        \rstDom{\projTan (\xmin) \Hess[\x] f (\xmin, \givenPrm)}{\spTan\xmin} \succ 0 \,.
    \end{equation}
\end{enumerate}
\end{assumption}
Finally, for partly smooth problems, IFT and linear convergence of PGD and APG also require a non-degeneracy assumption to hold which ensures that the solution stays in $\manif$ for any $\u$ near $\givenPrm$.
\begin{assumption}[Non-degeneracy] \label{ass:ND}
The non-degeneracy condition is satisfied at $(\xmin, \givenPrm)\in \manif\times\setPrm$, that is,
\begin{equation} \tag{ND} \label{eq:ND}
    0 \in \ri \partial_{\x} F (\xmin, \givenPrm) \,.
\end{equation}
\end{assumption}
\subsection{Implicit Function Theorem}
The following theorem is thanks to \cite{Lew02} which furnishes a $C^1$-smooth solution map for \eqref{prob:min:comp} near $\givenPrm$. The expression for the derivative can be found in \cite{VDP+17, MO24}.
\begin{theorem}[Differentiation of Solution Map] \label{thm:IFT}
Let $f$ and $g$ satisfy Assumption~\ref{ass:CPSO} and $(\xmin, \givenPrm)\in\manif\times\setPrm$ be such that Assumptions~\ref{ass:RPD}\ref{itm:RPD-i} and \ref{ass:ND} are satisfied. Then there exist an open neighbourhood $\neighbourhoodPRM\subset\setPrm$ of $\givenPrm$ and a continuously differentiable mapping $\map{\argmap}{\neighbourhoodPRM}{\manif}$ such that for all $\u\in \neighbourhoodPRM$,
\begin{enumerate}[label=(\roman*)]
    \item $\argmap(\u)$ is the unique minimizer of $\rstDom{F(\cdot, \u)}{\manif}$, \label{itm:IFT:i}
    \item \eqref{eq:ND} and \eqref{eq:RPD-i} are satisfied at $(\argmap (\u), \u)$, and \label{itm:IFT:ii}
    \item the derivative of $\argmap$ is given by
    \begin{equation} \label{eq:IFT}
        D \argmap (\u) = -\Hess[\manif] F (\argmap (\u), \u)^{\dagger} \derivu \grad[\manif] F (\argmap (\u), \u) \,,
    \end{equation}
    where $\Hess[\manif] F (\argmap (\u), \u)^{\dagger}$ denotes the pseudoinverse of $\Hess[\manif] F (\argmap (\u), \u)$. \label{itm:IFT:iii}
\end{enumerate}
\end{theorem}
\subsection{The Extrapolated Proximal Gradient Algorithm}
Although Theorem~\ref{thm:IFT} is not practical when computing the derivative of the solution map, it allows us to differentiate the update mappings of PGD and APG as we will see shortly. Algorithm~\ref{alg:EPG} provides the update procedure for proximal gradient with extrapolation where $\bwd$ is defined in \eqref{eq:prox}. This algorithm reduces to (i) PGD when $\beta_k = 0$ for all $k\in\N$, and (ii) APG when the sequence $\beta_k$ ensures acceleration in PGD \cite{BT09,CD15}.
\begin{algorithm} \ \label{alg:EPG}
    \begin{itemize}
        \item \key{Initialization:} $\xz=\xzm\in \spVar$, $\u\in\spPrm$, $0 < \sslow\leq\ssup < 2/L$.
        \item \key{Parameter:} $(\alpha_k)_{k\in\N} \in [\sslow, \ssup]$ and $(\beta_k)_{k\in\N} \in [0, 1]$.
        \item \key{Update $k\geq 0$:}
        \begin{equation} \label{itr:APG} \tag{APG}
            \begin{aligned}
                \yk &\coloneqq (1+\beta_k ) \xk - \beta_k \xkm \\
                \wk &\coloneqq \yk - \alpha_k \grad[\x] f(\yk, \u) \\
                \xkp &\coloneqq \bwd[k] (\wk, \u) \,.
            \end{aligned}
        \end{equation}
    \end{itemize}
\end{algorithm}

\subsection{Proximal Gradient Descent}
We first extend our results from Section~\ref{sec:FPI:k} and provide convergence rate guarantees for AD of PGD. For a given step size $\alpha>0$, we write the update mapping for PGD in a more compact manner through the map $\map{\pgd}{\spVar\times\setPrm}{\spVar}$ defined by 
\begin{equation} \label{eq:FixMap:PGD}
    \pgd (\x, \u) \coloneqq \bwd (\x - \alpha \grad[\x] f (\x, \u), \u) \,,
\end{equation}
Liang et al.~\cite{LFP14} established that under Assumptions~\ref{ass:CPSO}, \ref{ass:RPD}\ref{itm:RPD-ii} and \ref{ass:ND}, all the iterates $\xk (\givenPrm)$ of PGD lie on $\manif$ after a finite number of iterations and exhibit a local linear convergence behaviour as summarized by the following lemma.
\begin{lemma}[Activity Identification and Linear Convergence of PGD] \label{lem:PGD}
Let $f$ and $g$ satisfy Assumption~\ref{ass:CPSO} and $(\xmin, \givenPrm)\in\manif\times\setPrm$ be such that Assumption~\ref{ass:ND} is satisfied. For $\alpha_k\in[\sslow, \ssup]$ and $\beta_k\coloneqq0$, let the sequence $\seq[k\in\N]{\xk (\givenPrm)}$ generated by Algorithm~\ref{alg:EPG} converges to $\xmin$. Then there exists $K\in\N$, such that $\xk (\givenPrm)\in\manif$ for all $k\geq K$. Moreover when Assumption~\ref{ass:RPD}\ref{itm:RPD-ii} is also satisfied and $\ssup<2 M/L^2$ where $M\coloneqq \lambda_{\max} (\projTan (\xmin) \grad[\x] f(\xmin, \givenPrm) \projTan (\xmin))$, then $\xk (\givenPrm)$ converge linearly to $\xmin$.
\end{lemma}
\begin{remark} \label{rem:PGD}
    The precise rate of convergence of PGD can be found in \cite[Theorem~3.1]{LFP14}.
\end{remark}
Mehmood and Ochs~\cite[Theorem~34]{MO24} provide a foundation for proving the following crucial differentiablity result for $\pgd$ by combining Theorem~\ref{thm:IFT} and Lemma~\ref{lem:PGD}. However, the result below is novel since it establishes differentiability of $\pgd[k]$ at the iterates $\xk (\givenPrm)$ of PGD for all $k\in\N$ and also shows that the derivative sequence $\deriv \pgd[k]$ is well-behaved.
\begin{theorem} \label{thm:PGD:Derv}
    Let $f$ and $g$ satisfy Assumption~\ref{ass:CPSO} and $(\xmin, \givenPrm)\in\manif\times\setPrm$ be such that Assumption~\ref{ass:ND} are satisfied. For $\alpha_k\in[\sslow, \ssup]$ and $\beta_k\coloneqq0$, let the sequence $\seq[k\in\N]{\xk\coloneqq\xk (\givenPrm)}$ generated by Algorithm~\ref{alg:EPG} converges to $\xmin$. Then there exists $K\in\N$, such that,
    \begin{enumerate}[label=(\roman*)]
        \item \label{itm:PGD:Derv:Itrs} the mapping $(\x, \u, \alpha) \mapsto \pgd(\x, \u)$ defined in \eqref{eq:FixMap:PGD} is $C^1$-smooth near $(\xk, \givenPrm, \alpha_k)$ and $(\xmin, \givenPrm, \alpha_k)$ for all $k\geq K$,
        \item \label{itm:PGD:Derv:Conv} $\seq[k\geq K]{\deriv \pgd[k] (\xk, \givenPrm) - \deriv \pgd[k] (\xmin, \givenPrm)}$ converges to $0$, and
        \item \label{itm:PGD:Derv:Conv:Linear} additionally, when $\manif$ is $C^3$-smooth, $\grad[\manif] g$, $\Hess[\x] f$, $\Hess[\manif] g$, $D_{\u} \grad[\x] f$, and $D_{\u} \grad[\manif] g$ are locally Lipschitz continuous near $(\xmin, \givenPrm)$ and $\seq[k\in\N]{\xk}$ converges linearly, then $\deriv \pgd[k] (\xk, \givenPrm) - \deriv \pgd[k] (\xmin, \givenPrm) = \O (\xk - \xmin)$.
    \end{enumerate}
\end{theorem}
\findProofIn{thm:PGD:Derv}
\begin{remark}
    The $C^3$-smoothness of $\manif$ is a sufficient condition for the local Lipschitz continuity of the Weingarten map $(\x, \v) \mapsto \wein[\x]{\cdot}{\v}$.
\end{remark}
\subsubsection{Implicit Differentiation}
Under Assumptions~\ref{ass:CPSO}--\ref{ass:ND}, Mehmood and Ochs~\cite[Theorem~34]{MO24} provided an IFT for the fixed-point equation of PGD which is more practical for using ID than \eqref{eq:IFT} and is restated below.
\begin{theorem} \label{thm:PGD:IFT}
    Let $f$ and $g$ satisfy Assumption~\ref{ass:CPSO} and $(\xmin, \givenPrm)\in\manif\times\setPrm$ be such that Assumptions~\ref{ass:RPD}\ref{itm:RPD-ii} and \ref{ass:ND} are satisfied. Then for any $\alpha\in[\sslow, \ssup]$, $\rho (D_{\x} \pgd (\xmin, \givenPrm) \projTan (\xmin)) < 1$. Additionally when Assumption~\ref{ass:RPD}\ref{itm:RPD-i} is also satisfied, the (possibly reduced) neighbourhood $\neighbourhoodPRM$ and the mapping $\argmap$ from Theorem~\ref{thm:IFT} satisfy $\x = \pgd (\x, \u)$ and
    \begin{equation} \label{eq:PGD:IFT}
        D \argmap (\u) = \left( \opid - D_{\x} \pgd (\x, \u) \projTan (\x) \right)^{-1} D_{\u} \pgd (\x, \u) \,,
    \end{equation}
    for all $\u\in \neighbourhoodPRM$ and $\x \coloneqq \argmap (\u)$.
\end{theorem}
\subsubsection{Automatic Differentiation}
Using Theorems~\ref{thm:PGD:Derv}, and \ref{thm:PGD:IFT} and our results from Section~\ref{sec:FPI:k}, we obtain the convergence (rate) guarantees for AD of PGD.
\begin{theorem} \label{thm:PGD:AD}
    Let $f$ and $g$ satisfy Assumption~\ref{ass:CPSO} and $(\xmin, \givenPrm)\in\manif\times\setPrm$ be such that Assumptions~\ref{ass:RPD} and \ref{ass:ND} are satisfied. Let $\alpha_k\in[\sslow, \ssup]$, $\beta_k \coloneqq0$ and the sequence $\seq[k\in\N]{\xk(\givenPrm)}$ generated by Algorithm~\ref{alg:EPG} converges to $\xmin$ with $\xz$ sufficiently close to $\xmin$. Then the sequence $\seq[k \in\N]{\deriv \xk (\givenPrm)}$ converges to $D \argmap(\givenPrm)$. Additionally $\seq[k \in\N]{\deriv \xk (\givenPrm)}$ converges linearly with rate $\max(q_{\x}, \limsup_{k\to\infty} \rho (\derivx \pgd[k] (\xmin, \givenPrm)))$ when $\manif$ is $C^3$-smooth, $\grad[\manif] g$, $\Hess[\x] f$, $\Hess[\manif] g$, $D_{\u} \grad[\x] f$, and $D_{\u} \grad[\manif] g$ are locally Lipschitz continuous near $(\xmin, \givenPrm)$ and $\seq[k\in\N]{\xk(\givenPrm)}$ converges with rate $q_{\x} < 1$.
\end{theorem}
\findProofIn{thm:PGD:AD}
\begin{remark} \label{rem:PGD:AD}
    \begin{enumerate}[label=(\roman*)]
        \item Because $\xz$ is not close enough to $\xmin$ in practice, one may resort to late-starting \cite[Section~1.4.2]{MO24}.
        \item In practice, we do not need $\xz$ to be close enough to $\xmin$ because even when the update map $\pgd[k]$ is not differentiable in the earlier iterations of Algorithm~\ref{alg:EPG}, the autograd libraries still yield a finite output as the derivative \cite{BP20, BP21}. Hence, AD of Algorithm~\ref{alg:EPG} can still recover a good estimate of $D\argmap (\u)$ as long as Algorithm~\ref{alg:EPG} is run for sufficiently large number of iterations \cite[Remark~45(i)]{MO24}.
        \item When $\alpha_k$ is generated through line-search methods, it also depends on $\u$ in a possibly non-differentiable way. Therefore, the total derivative of $\xk$ with respect to $\u$ may not make sense. However computing $\deriv\xk (\u)$ by ignoring this dependence --- for instance, in practice, through routines like \verb|stop_gradient| \cite{BFH+18,ABC+16} and \verb|detach| \cite{PGM+19} --- the true derivative is still recovered in the limit provided that conditions of Theorem~\ref{thm:PGD:AD} are met. This fact was first explored in \cite{GBC+93}.
    \end{enumerate}
\end{remark}

\subsection{Accelerated Proximal Gradient}
We similarly compute AD of APG and show (linear) convergence of the corresponding derivative iterates. Given step size $\alpha > 0$ and extrapolation parameter $\beta\in[0, 1]$, we define the update mapping $\map{\apg}{\spVar\times\spVar\times\spPrm}{\spVar\times\spVar}$ by
\begin{equation} \label{eq:FixMap:APG}
    \apg (\z, \u) \coloneqq \left(\pgd \left(\x_1 + \beta (\x_1 - \x_2), \u \right), \x_1 \right) \,,
\end{equation}
for $\z\coloneqq(\x_1, \x_2)\in\spVar\times\spVar$. Just like PGD, APG also exhibit activity identification property and local linear convergence under Assumptions~\ref{ass:CPSO}, \ref{ass:RPD}\ref{itm:RPD-ii}, and \ref{ass:ND} \cite{LFP17}.

\begin{lemma}[Activity Identification and Linear Convergence of APG] \label{lem:APG}
Let $f$ and $g$ satisfy Assumption~\ref{ass:CPSO} and $(\xmin, \givenPrm)\in\manif\times\setPrm$ be such that Assumption~\ref{ass:ND} is satisfied. For $\alpha_k\in[\sslow, \ssup]$ and $\beta_k\in[0, 1]$, let the sequence $\seq[k\in\N]{\xk (\givenPrm)}$ generated by Algorithm~\ref{alg:EPG} converges to $\xmin$. Then there exists $K\in\N$, such that $\xk (\givenPrm)\in\manif$ for all $k\geq K$. Moreover when Assumption~\ref{ass:RPD}\ref{itm:RPD-ii} is also satisfied, $\alpha_k \to \alpha_*$ and $\beta_k \to \beta_*$ such that $-1/(1 + 2\beta_*) < \lambda_{\min} (\derivx \pgd[*] (\xmin, \givenPrm))$, then $\xk (\givenPrm)$ converge linearly to $\xmin$ with rate $\rho (\derivz \apg[*] (\xmin, \xmin, \givenPrm) \projTan (\xmin, \xmin))$.
\end{lemma}
Using Theorem~\ref{thm:IFT}, and Lemma~\ref{lem:APG}, we can differentiate $\apg[k]$ near $(\xmin, \xmin, \givenPrm)$ for all $k\in\N$. The following result is mostly derived from \cite[Theorem~39]{MO24}.
\begin{theorem} \label{thm:APG:Derv}
    Let $f$ and $g$ satisfy Assumption~\ref{ass:CPSO} and $(\xmin, \givenPrm)\in\manif\times\setPrm$ be such that Assumption~\ref{ass:ND} are satisfied. For $[\sslow, \ssup] \ni \alpha_k \to \alpha_*$ and $[0, 1]\ni\beta_k \to \beta_*$, let the sequence $\seq[k\in\N]{\xk\coloneqq\xk (\givenPrm)}$ generated by Algorithm~\ref{alg:EPG} converges to $\xmin$. Then there exist $\neighbourhoodAPG[*]\in\neighbourhood[(\xmin, \xmin, \givenPrm, \alpha_*)]$ and $K\in\N$, such that
    \begin{enumerate}[label=(\roman*)]
        \item \label{itm:APG:Derv:Itrs} the mapping $(\z, \u, \alpha) \mapsto \apg(\z, \u)$ defined in \eqref{eq:FixMap:APG} is $C^1$-smooth on $\neighbourhoodAPG[*]$ and $(\zk, \givenPrm, \alpha_k) \in \neighbourhoodAPG[*]$ for all $k\geq K$,
        \item \label{itm:APG:Derv:Conv} $\seq[k\geq K]{\deriv \apg[k] (\zk, \givenPrm) - \deriv \apg[k] (\z^*, \givenPrm)}$ converges to $0$, and
        $\seq[k\geq K]{\deriv \apg[k] (\z^*, \givenPrm)}$ converges to $\deriv \apg[*] (\z^*, \givenPrm)$, and
        \item \label{itm:APG:Derv:Conv:Linear} additionally, when $\manif$ is $C^3$-smooth, $\grad[\manif] g$, $\Hess[\x] f$, $\Hess[\manif] g$, $D_{\u} \grad[\x] f$, and $D_{\u} \grad[\manif] g$ are locally Lipschitz continuous near $(\xmin, \givenPrm)$ and $\seq[k\in\N]{\xk}$ converges linearly, then $\deriv \apg[k] (\zk, \givenPrm) - \deriv \apg[k] (\z^*, \givenPrm) = \O (\xk - \xmin)$,
    \end{enumerate}
    where $\z\coloneqq(\x_1, \x_2)$, $\zk \coloneqq (\xk, \xkm)$, and $\z^*\coloneqq (\xmin, \xmin)$
\end{theorem}
\findProofIn{thm:APG:Derv}
\subsubsection{Implicit Differentiation}
Similarly Theorems~\ref{thm:IFT}, and \ref{thm:APG:Derv} can be used to yield an IFT for the fixed-point equation of APG \cite[Theorem~39]{MO24}, which we recall below.
\begin{theorem} \label{thm:APG:IFT}
    Let $f$ and $g$ satisfy Assumption~\ref{ass:CPSO} and $(\xmin, \givenPrm)\in\manif\times\setPrm$ be such that Assumptions~\ref{ass:RPD}\ref{itm:RPD-ii} and \ref{ass:ND} are satisfied. Then for any $\alpha\in[\sslow, \ssup]$ and $\beta\in[0, 1]$ with $-1/(1 + 2\beta) < \lambda_{\min} (D_{\x} \pgd (\xmin, \u))$, we have $\rho (D_{\z} \apg (\xmin, \xmin, \givenPrm) \projTan (\xmin, \xmin)) < 1$. Additionally when Assumption~\ref{ass:RPD}\ref{itm:RPD-i} is also satisfied, the (possibly reduced) neighbourhood $\neighbourhoodPRM$ and the mapping $\argmap$ from Theorem~\ref{thm:IFT} satisfy $\z = \apg (\z, \u)$ and
    \begin{equation} \label{eq:APG:IFT}
        \begin{bmatrix}
            \deriv \argmap (\u) \\
            \deriv \argmap (\u)
        \end{bmatrix} = \left( \opid - \derivz \apg (\z, \u) \projTan (\z)\right)^{-1} \derivu \apg (\z, \u) \,,
    \end{equation}
    for all $\u\in \neighbourhoodPRM$ and $\z \coloneqq (\argmap (\u), \argmap (\u))$.
\end{theorem}
\subsubsection{Automatic Differenitation}
Mehmood and Ochs~\cite[see~Theorem~44]{MO24} established the convergence of the derivative iterates of APG. We strengthen their results by providing convergence rate guarantees in our final result below.
\begin{theorem} \label{thm:APG:AD}
    Let $f$ and $g$ satisfy Assumption~\ref{ass:CPSO} and $(\xmin, \givenPrm)\in\manif\times\setPrm$ be such that Assumption~\ref{ass:ND} are satisfied. For $[\sslow, \ssup] \ni \alpha_k \to \alpha_*$ and $[0, 1]\ni\beta_k \to \beta_*$, let the sequence $\seq[k\in\N]{\xk (\givenPrm)}$ generated by Algorithm~\ref{alg:EPG} converges to $\xmin$ with $\xz$ sufficiently close to $\xmin$. Then the sequence $\seq[k \in\N]{\deriv \xk (\givenPrm)}$ converges to $D \argmap(\givenPrm)$. Additionally $\seq[k \in\N]{\deriv \xk (\givenPrm)}$ converges linearly with rate $\max(q_{\x}, \rho (\derivz \apg[*] (\xmin, \givenPrm)))$ when $\manif$ is $C^3$-smooth, $\grad[\manif] g$, $\Hess[\x] f$, $\Hess[\manif] g$, $D_{\u} \grad[\x] f$, and $D_{\u} \grad[\manif] g$ are locally Lipschitz continuous near $(\xmin, \givenPrm)$ and $\seq[k\in\N]{\xk(\givenPrm)}$ converges with rate $q_{\x} < 1$.
\end{theorem}
\findProofIn{thm:APG:AD}
\begin{remark}
    The arguments made in Remark~\ref{rem:PGD:AD} naturally extend to Theorem~\ref{thm:APG:AD}.
\end{remark}

\section{Experiments} \label{sec:exp}

\begin{figure}
    \centering
    \includegraphics[width=0.96\linewidth]{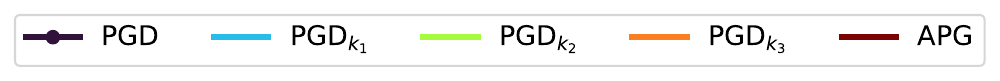}
    \includegraphics[width=0.48\linewidth]{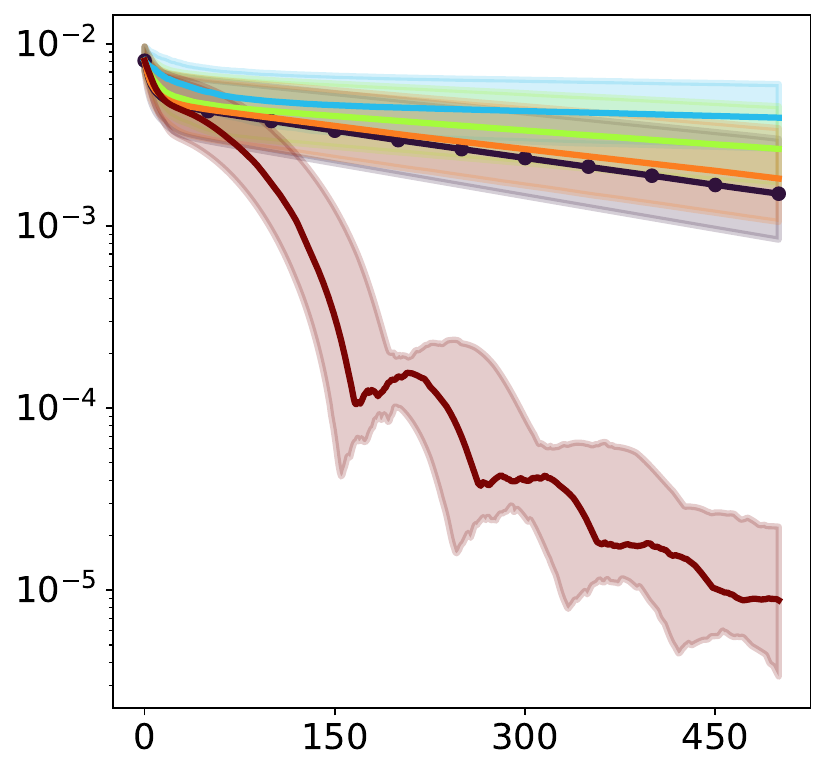}
    \includegraphics[width=0.48\linewidth]{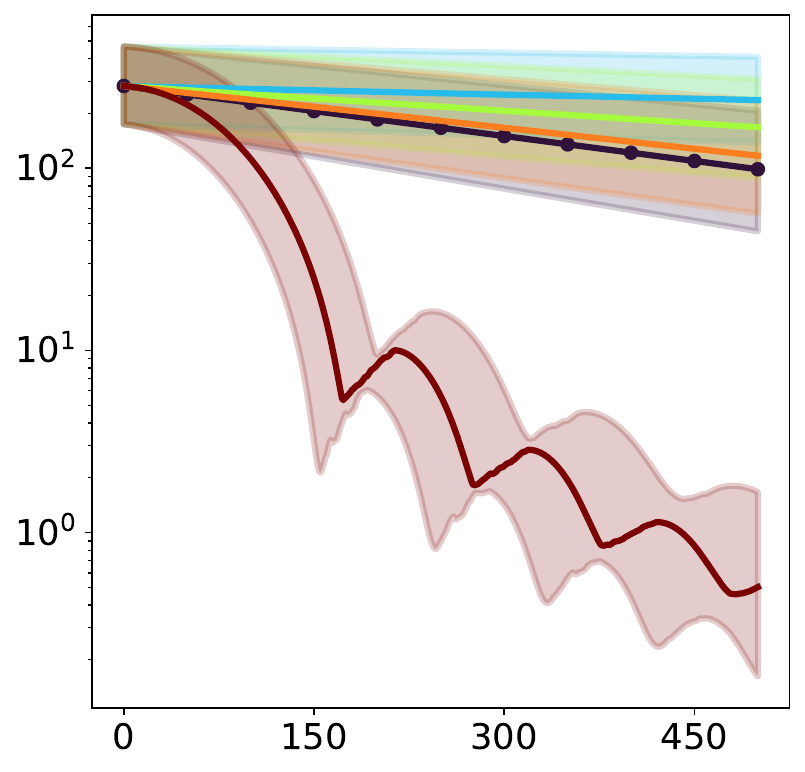}
    \includegraphics[width=0.48\linewidth]{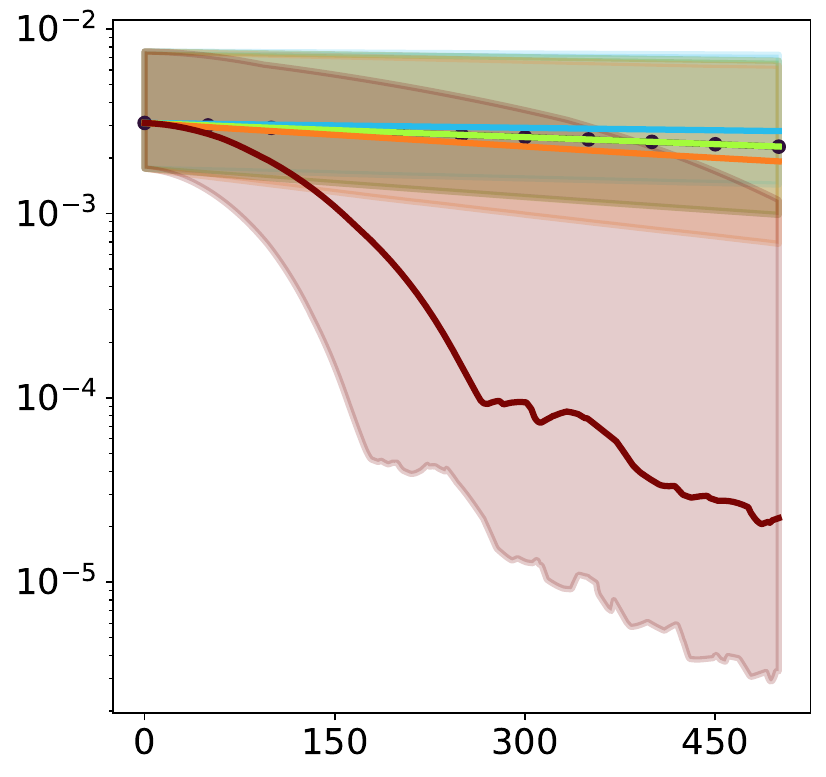}
    \includegraphics[width=0.48\linewidth]{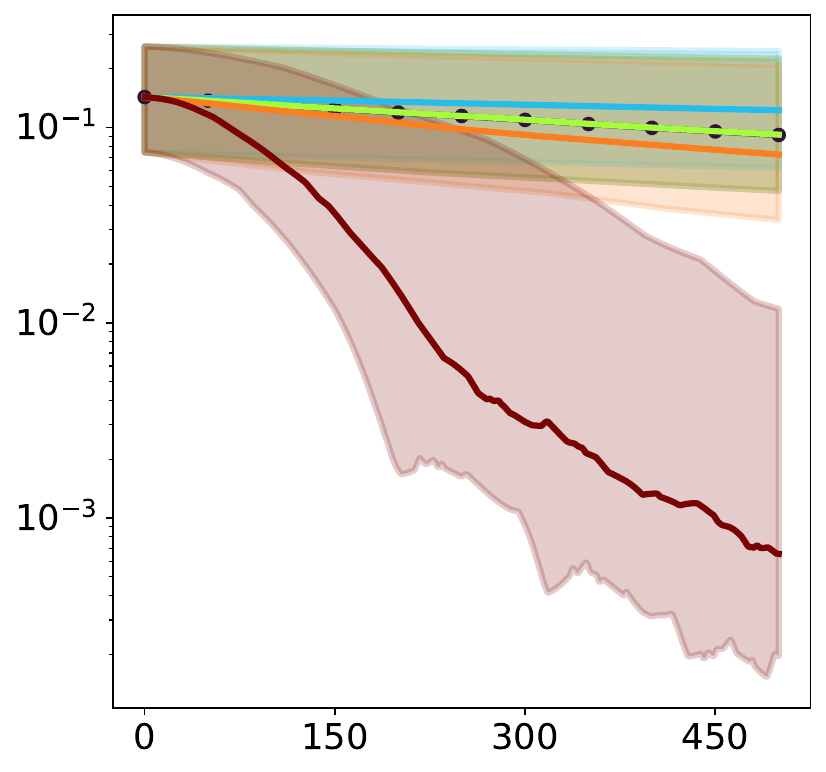}
    \caption{Error plots of iterates (left column) of PGD and APG for Logistic (top row) and Lasso (bottom row) Regression along with their derivative iterates (right column). Similarity in the convergence rates of the original and the derivative iterates is clearly visible.}
    \label{fig:exp}
\end{figure}
To test our results, we provide numerical demonstration on one smooth and one non-smooth example from classical Machine Learning. These include logistic regression with $\ell_2$ regularization, that is,
\begin{equation} \label{prob:logistic}
    \min_{\x} \frac{1}{M} \sum_{i=1}^{M} \log \left(1 + \exp (-b_i \mathbf{a}_i^T \x)\right) + \frac{1}{2} \lambda\norm[2]{\x}^2 \,,
\end{equation}
with parameters $\u \coloneqq (A, \lambda)$ and linear regression with $\ell_1$ regularization, that is,
\begin{equation} \label{prob:lasso}
    \min_{\x} \frac{1}{2} \norm[2]{A\x - \b} + \lambda\norm[1]{\x} \,,
\end{equation}
with parameters $\u \coloneqq (A, \b, \lambda)$. All the parameters are selected in such a way that Assumptions~\ref{ass:CPSO}--\ref{ass:ND} are satisfied. 

We solve the two problems through PGD with four different choices of step sizes and APG with fixed step size and $\beta_k \coloneqq (k-1) / (k+5)$ (depicted by $\APG$ in Figure~\ref{fig:exp}). This generates five different algorithm sequences $\seq[k\in\N]{\xk (\u)}$ for each problem. The four step size options for PGD include $\alpha_k\coloneqq\alpha_*$ (depicted by $\PGD$ in Figure~\ref{fig:exp}), $\alpha_k \sim U (0, \frac{2}{3L})$ (depicted by $\PGD_{k_1}$), $\alpha_k \sim U (\frac{2}{3L}, \frac{4}{3L})$ (depicted by $\PGD_{k_2}$), and $\alpha_k \sim U (\frac{4}{3L}, \frac{2}{L})$ (depicted by $\PGD_{k_3}$) where $\alpha_*$ is the optimal step size for each problem. We solve each problem for $50$ different instances of the parameters and plot the median error. For each problem, we initialize $\xz \in B_{10^{-2}} (\xmin)$ by first solving it partially.

In Figure~\ref{fig:exp}, the left column shows the median error plots of the five algorithms and the right column shows the errors of the corresponding derivatives with the same colour. The top and bottom rows correspond to the error plots for \eqref{prob:logistic} and \eqref{prob:lasso} respectively. The figure clearly shows that the derivative error decays as fast as the algorithm error for each problem.

\section{Conclusion}
\label{sec:Conc}

We applied automatic differentiation on the iterative processes with time-varying update mappings. We strengthened a previous result \cite{Bec94} with convergence rate guarantee and extended them to a new setting. As an example for each setting, we adapted our results to proximal gradient descent with variable step size and its accelerated counterpart, that is, FISTA. We showed that the convergence rate of the algorithm is simply mirrored in its derivative iterates which was supported through experiments on toy problems.


\appendix

\section{Preliminaries \& Related Work} \label{sec:pre:supp}
Before we move on to the proofs of our main results, we present some preliminary results  which will be useful later. We also provide a recap of the results of \cite{Bec94} for a better understanding of our work.
\subsection{Matrix Analysis} \label{ssec:pre:Mat:Ana}
This section provides a few preliminary results on Matrix Analysis. We start by recalling a classical result from Linear Algebra which forms the foundation for the implicit differentiation applied to \eqref{eq:FPE} when combined with the Implicit Function Theorem \cite[Theorem~1B.1]{DR09}.
\begin{lemma} \label{lem:spec_rad:inv_op}
    For any linear operator $\map{B}{\spVar}{\spVar}$ with $\rho (B) < 1$, the linear operator $\opid - B$ is invertible.
\end{lemma}
We now provide some results which will be used in the proofs of convergence of automatic differentiation of the fixed-point iterations. In particular, we provide convergence guarantees of sequences generated through linear iterative procedures including \eqref{itr:LGFPI}, and \eqref{itr:LGFPI:modified} which bear strong resemblance with \eqref{itr:D:FPI} and \eqref{itr:D:FPI:k} (see the proofs of Theorems~\ref{thm:D:FPI:conv} and \ref{thm:D:FPI:modified:conv} for a more detailed comparison). These convergence results are mainly derived from \cite[Section~2.1.2,~Theorem~1]{Pol87}, \cite[Proposition~2.7]{Rii20}, and \cite[Theorem~8]{MO24}. The following theorem will be used to prove convergence results in Section~\ref{sec:FPI:k}.
\begin{theorem} \label{thm:LGFPI:modified:conv}
    Let $\seq[k\in\N_0]{B_k}$, $\seq[k\in\N_0]{C_k}$ and $\seq[k\in\N_0]{\dk}$ be sequences in $\spLin(\spVar, \spVar)$, $\spLin(\spVar, \spVar)$ and $\spVar$ respectively such that $C_k \to 0$, $\dk \to 0$ and $\rho\coloneqq \limsup_{k\to\infty} \norm{B_k} < 1$ where $\norm{\cdot}$ is a norm on $\spLin(\spVar, \spVar)$ induced by some vector norm. Then the sequence $\seq[k\in\N_0]{\ek}$, with $\ez\in \spVar$, generated by
    \begin{equation} \label{itr:LGFPI:modified}
        \ekp \coloneqq B_k \ek + C_k\ek + \dk \,,
    \end{equation}
    converges to $0$. The convergence is linear when $C_k$ and $\dk$ converge linearly with rates $q_{C}$ and $q_{\d}$ respectively. In fact, for all $\eps \in (0, 1 - \rho)$, there exist $C_1(\eps)$, $C_2(\eps)$ and $K\in\N$, such that for all $k\geq K$, we have
    \begin{equation} \label{eq:LGFPI:conv_rates:modified}
        \norm{\ek} \leq C_1 (\eps) (k-K) q^{k-K} + C_2 (\eps) (\rho + \eps)^{k-K} \,,
    \end{equation}
    where $q\coloneqq \max (\rho + \eps, q_{C}, q_{\d})$.
\end{theorem}
\begin{proof}
    From classical analysis, given $\eps \in (0, 1-\rho)$, there exists $K\in\N$ such that $\norm{B_k} < \rho + \eps/2$ for all $k\geq K$. Also, with $K$ large enough we have $\norm{C_k\ek} / \norm{\ek} \leq \eps/2$ since $C_k\to0$. Therefore, from \eqref{itr:LGFPI:modified}, we get
    \begin{equation*}
        \begin{aligned}
            \norm{\ekp} &\leq (\rho + \epsilon/2) \norm\ek + \norm{C_k\ek} + \norm\dk \\
            &\leq (\rho + \epsilon) \norm\ek + \norm\dk \,.
        \end{aligned}
    \end{equation*}
    The convergence of $\ek$ then follows from the arguments following Equation~34 in the proof of Theorem~8 in \cite{MO24}.

    For the rate of convergence, we note that $\epsk \coloneqq C_k\ek + \dk$ converges linearly to $0$ with rate $\max(q_{C}, q_{\d})$ and for $K\in\N$ sufficiently large we can write $\norm{\epsk} \leq \max(q_{C}, q_{\d})^{k-K} \norm{\epsK}$ for all $k\geq K$. Thus, for any $k\geq K$, we expand the expression $\ekp = B_k\ek + \epsk$ to obtain
    \begin{equation*}
        \ekp = B_{k:K}\eK + \sum_{i=K}^{k} B_{k:i+1} \epsi
    \end{equation*}
    where $B_{k:i}$ denotes the ordered product $B_k B_{k-1} \cdots B_{i+1} B_i$ when $i\leq k$ and identity operator $\opid$ when $i > k$. Observe that, for any $\eps \in (0, 1-\rho)$ and $K\in\N$ large enough, $\norm{B_k} < \rho + \eps$ for all $k\geq K$ and therefore for any $K\leq i\leq k+1$, we have $\norm{B_{k:i}} \leq (\rho+\eps)^{k-i+1}$. Therefore, setting $q \coloneqq \max (\rho + \eps, q_C, q_{\d})$, we end up with
    \begin{equation} \label{ineq:LGFPI:lin:modified}
        \begin{aligned}
            \norm{\ekp} &\leq \norm{B_{k:K}} \norm{\eK} + \sum_{i=K}^{k} \norm{B_{k:i+1}} \norm{\epsi} \\
            &\leq \norm{\eK} (\rho + \eps)^{k-K+1} + \sum_{i=K}^{k} (\rho + \eps)^{k-i} \norm{\epsK} \max(q_C, q_{\d})^{i-K} \\
            &\leq \norm{\eK} (\rho + \eps)^{k-K+1} + \norm{\epsK} \sum_{i=K}^{k} q^{k-K} \\
            &\leq \norm{\eK} (\rho + \eps)^{k-K+1} + \norm{\epsK} (k-K+1) q^{k-K} \,.
        \end{aligned}
    \end{equation}
\end{proof}
The result below is a direct consequence of the above theorem and will find its use when proving Corollary~\ref{cor:D:FPI:modified:conv:(linear)}.
\begin{corollary} \label{cor:LGFPI:modified:conv}
    The conclusion of Theorem~\ref{thm:LGFPI:modified:conv} also holds when the assumptions on the sequence $\seq[k\in\N_0]{B_k}$ are replaced with $B_k\to B\in\spLin (\spVar, \spVar)$ and $\rho (B) < 1$.
\end{corollary}
\begin{proof}
    This is a special case of Theorem~\ref{thm:LGFPI:modified:conv} because for any $\delta \in (0, 1-\rho (B))$, there exists a norm on $\spLin (\spVar, \spVar)$ induced by some vector norm, both denoted by $\norm[\delta]{\cdot}$, such that $\norm[\delta]{B}\leq\rho (B) + \delta$ and $\rho\coloneqq\lim_{k\to\infty} \norm[\delta]{B_k} = \norm[\delta]{B} \leq \rho (B) + \delta<1$.
\end{proof}
The following result is a special case of the setting of Corollary~\ref{cor:LGFPI:modified:conv} and is straightforward to show.
\begin{corollary} \label{cor:cor:LGFPI:modified:conv}
    The conclusion of Theorem~\ref{thm:LGFPI:modified:conv} also holds when the assumptions on the sequence $\seq[k\in\N_0]{B_k}$ are replaced with $B_k \coloneqq B\in\spLin (\spVar, \spVar)$ for all $k\in\N$ and $\rho (B) < 1$.
\end{corollary}
Using Corollary~\ref{cor:cor:LGFPI:modified:conv}, we can prove the following statement which will be useful in the convergence proofs of Section~\ref{ssec:pre:FPI:k:Beck:supp}.
\begin{theorem} \label{thm:LGFPI:conv}
Let $\seq[k\in\N_0]{B_k}$ and $\seq[k\in\N_0]{\bk}$ be sequences in $\spLin(\spVar, \spVar)$ and $\spVar$ with limits $B$ and $\b$, respectively. If $\rho=\rho (B) < 1$, the sequence $\seq[k\in\N_0]{\xk}$, with $\xz\in \spVar$, generated by
\begin{equation} \label{itr:LGFPI}
    \xkp \coloneqq B_k \xk + \bk \,,
\end{equation}
converges to $\x\coloneqq(\opid - B)^{-1} \b$. The convergence is linear when $\seq[k\in\N_0]{B_k}$ and $\seq[k\in\N_0]{\bk}$ converge linearly with rates $q_B$ and $q_{\b}$, respectively. In fact, for all $\delta \in (0, 1 - \rho)$, there exist $C_1(\delta)$, $C_2(\delta)$ and $K\in\N$, such that for all $k\geq K$, we have
\begin{equation} \label{eq:LGFPI:conv_rates}
    \norm{\xk - \x} \leq C_1 (\delta) (k-K) q^{k-K} + C_2 (\delta) (\rho + \delta)^{k-K} \,,
\end{equation}
where $q\coloneqq \max (\rho + \delta, q_B, q_{\b})$.
\end{theorem}
\begin{proof}
The expression $\x = (I-B)^{-1}\b$ is well-defined thanks to Lemma~\ref{lem:spec_rad:inv_op} and solves the linear equation $\x = B \x + \b$ for $\x$. By setting $\ek\coloneqq\xk - \x$, $C_k\coloneqq B_k - B$, and $\dk\coloneqq (B_k - B)\x + \bk - \b$, we obtain,
\begin{equation*}
    \begin{aligned}
        \ekp &= (B_k \xk + \bk) - (B\x + \b) \\
        &= B \ek + C_k \ek + \dk \,.
    \end{aligned}
\end{equation*}
Since the above recursion matches \eqref{itr:LGFPI:modified} and the setting of Corollary~\ref{cor:cor:LGFPI:modified:conv} applies, the result follows.
\end{proof}

\subsection{Differentiation of Iteration-Dependent Algorithms: Classical Results} \label{ssec:pre:FPI:k:Beck:supp}

In this section, we briefly recap the results of \cite{Bec94} in a way that aligns with those of Section~\ref{sec:FPI:k}, allowing for a clearer comparison between the two sets of results. We first lay down the assumptions on the sequence of mappings $\fixMap[k]$ in \eqref{itr:FPI:k} and the mapping $\fixMap$ in \eqref{eq:FPE} which will ensure the convergence of the sequence $\deriv \xk (\u)$ generated by \eqref{itr:D:FPI:k}. The main requirement in the work of \cite{Bec94} is that of the pointwise convergence of the sequence $\deriv \fixMap[k]$ to $\deriv \fixMap$.

\subsubsection{Problem Setting}

Given $\givenPrm\in\spVar$, we assume that $\xmin$ solves \eqref{eq:FPE} for $\x$ with $\u=\givenPrm$ and is the desired limit of $\xk (\givenPrm)$ generated by \eqref{itr:FPI:k}. Furthermore, we assume that for all $k$, $\fixMap[k]$ and $\fixMap$ are $C^1$-smooth near $(\xmin, \givenPrm)$ and the contraction property holds for $\derivx\fixMap$ at $(\xmin, \givenPrm)$. In particular, given $(\xz, \xmin, \givenPrm) \in \spVar\times\spVar\times\spPrm$, $\neighbourhoodFM \in \neighbourhood[(\xmin, \givenPrm)]$, and $\seq[k\in\N_0]{\fixMap[k]}$, we assume that the following assumption holds.
\begin{assumption} \label{ass:basic}
    \begin{enumerate}[label=(\roman*)]
        \item \label{itm:basic:C1} $\rstDom{\fixMap[k]}{\neighbourhoodFM}$ and $\rstDom{\fixMap}{\neighbourhoodFM}$ are $C^1$-smooth for all $k$,
        \item \label{itm:basic:fixed:point} $\xmin = \fixMap (\xmin, \givenPrm)$,
        \item \label{itm:basic:spectral:radius} $\rho (\derivx \fixMap (\xmin, \givenPrm)) < 1$, and
        \item \label{itm:basic:iterates} $\xk (\givenPrm)$ generated by \eqref{itr:FPI:k} has limit $\xmin$ such that $(\xk (\givenPrm), \givenPrm) \in \neighbourhoodFM$ for all $k\in\N$.
    \end{enumerate}
\end{assumption}

\subsubsection{Implicit Differentiation} \label{sssec:Beck:ID:supp}
From Assumption~\ref{ass:basic}, Lemma~\ref{lem:spec_rad:inv_op}, and \cite[Theorem~1B.1]{DR09}, we obtain the Implicit Function Theorem for \eqref{eq:FPE}.
\begin{theorem}[Implicit Function Theorem] \label{thm:IFT:FPE}
    For some $(\xmin, \givenPrm)$, let $\fixMap$ be $C^1$-smooth near $(\xmin, \givenPrm)$, $\xmin = \fixMap (\xmin, \givenPrm)$ and $\rho (\derivx \fixMap (\xmin, \givenPrm)) < 1$. Then $\exists \ \neighbourhoodPRM\in\neighbourhood[\givenPrm]$ and a $C^1$-smooth mapping $\map{\argmap}{\neighbourhoodPRM}{\spVar}$ such that $\forall \u\in \neighbourhoodPRM$, $\argmap(\u) = \fixMap(\argmap(\u), \u)$, and
    \begin{equation} \label{eq:FPE:IFT}
        \deriv\argmap(\u) = \big( \opid - \derivx \fixMap (\argmap(\u), \u) \big)^{-1} \derivu \fixMap (\argmap(\u), \u) \,.
    \end{equation}
\end{theorem}

\subsubsection{Automatic Differentiaion} \label{sssec:Beck:AD:supp}
As stated in Section~\ref{ssec:pre:Mat:Ana}, the update procedure to generate the derivative iterates in \eqref{itr:D:FPI:k} takes after the iterative process defined by \eqref{itr:LGFPI} in Theorem~\ref{thm:LGFPI:conv}. That is, for some $\givenPrm\in\spPrm$ and $\ud\in\spPrm$, if we set $B_k \coloneqq\derivx\fixMap[k](\xk (\givenPrm), \givenPrm)$ and $\bk \coloneqq\derivu\fixMap[k](\xk (\givenPrm), \givenPrm) \ud$, the resulting sequence is $\yk = \deriv\xk(\givenPrm)\ud$. Similarly, the limit $\y^*\coloneqq(\opid - B)^{-1} \b$ of the sequence $\yk$ generated by \eqref{itr:LGFPI} matches $\deriv\argmap (\givenPrm)\ud$ from \eqref{eq:FPE:IFT}, if we set $B \coloneqq\derivx\fixMap(\argmap (\givenPrm), \givenPrm)$ and $\b\coloneqq\derivu\fixMap(\argmap (\givenPrm), \givenPrm) \ud$. One way to prove the convergence of $\deriv\xk (\givenPrm)\ud$ to $\deriv\argmap(\givenPrm)\ud$ is by asserting that $\deriv\fixMap[k] (\xk (\givenPrm), \givenPrm)$ converges to $\deriv\fixMap (\argmap (\givenPrm), \givenPrm)$.
\begin{assumption} \label{ass:convergence}
    The sequence $\seq[k\in\N_0]{\deriv\fixMap[k] (\xk (\givenPrm), \givenPrm)}$ converges to $ \deriv\fixMap (\xmin, \givenPrm)$.
\end{assumption}
\begin{remark} \label{rem:convergence}
    When the sequence of functions $\seq[k\in\N]{\deriv\fixMap[k]}$ is equicontinuous at $(\xmin, \givenPrm)$ and has a pointwise limit $\deriv\fixMap$ near $(\xmin, \givenPrm)$, then Assumption~\ref{ass:convergence} naturally holds. 
\end{remark}
The main result of \cite{Bec94} can then be stated below.
\begin{theorem} \label{thm:D:FPI:conv}
    Let $(\xz, \xmin, \givenPrm)$ be such that Assumptions~\ref{ass:basic} and \ref{ass:convergence} are satisfied by $\fixMap[k]$ and $\fixMap$. Then the sequence $\seq[k\in\N_0]{\deriv\xk (\givenPrm)}$ generated by \eqref{itr:D:FPI:k} converges to $\deriv\argmap (\givenPrm)$.
\end{theorem}
\begin{proof}
    The proof is a direct consequence of Assumptions~\ref{ass:basic} and \ref{ass:convergence}, Theorem~\ref{thm:LGFPI:conv} and the arguments following Theorem~\ref{thm:LGFPI:conv}.
\end{proof}
For linear convergence we require a stronger assumption on $D\fixMap[k]$, that is its linear convergence as given below.
\begin{assumption} \label{ass:linear:convergence}
    The sequence $\seq[k\in\N_0]{\xk (\givenPrm)}$ converges linearly to $\xmin$ and
    \begin{equation} \label{ineq:Lipschitz:Like}
        \deriv\fixMap[k] (\xk (\givenPrm), \givenPrm) - \deriv\fixMap (\xmin, \givenPrm) = \O (\xk (\givenPrm) - \xmin) \,.
    \end{equation}
\end{assumption}

\begin{theorem} \label{thm:D:FPI:conv:linear}
    Let $(\xz, \xmin, \givenPrm)$ be such that Assumptions~\ref{ass:basic} and \ref{ass:linear:convergence} are satisfied by $\fixMap[k]$ and $\fixMap$. Then the sequence $\seq[k\in\N_0]{\deriv\xk (\givenPrm)}$ generated by \eqref{itr:D:FPI:k} converges linearly to $\deriv\argmap (\givenPrm)$. In particular, when the sequence $\seq[k\in\N_0]{\xk}$ converges with rate $q_{\x} < 1$, then for all $\delta \in (0, 1 - \rho)$, there exist $C_1(\delta)$, $C_2(\delta)$ and $K\in\N$, such that for all $k\geq K$, we have
    \begin{equation} \label{eq:D:FPI:conv_rates}
        \norm{\deriv \xk (\givenPrm) - \deriv \argmap (\givenPrm)} \leq C_1 (\delta) (k-K) q^{k-K} + C_2 (\delta) (\rho + \delta)^{k-K} \,,
    \end{equation}
    where $\rho \coloneqq \rho (\derivx \fixMap (\xmin, \givenPrm))$ and $q \coloneqq \max(\rho + \delta, q_{\x})$.
\end{theorem}
\begin{proof}
    Assumption~\ref{ass:linear:convergence} asserts that the sequences $\derivx \fixMap[k] (\xk (\givenPrm), \givenPrm)$ and $\derivu \fixMap[k] (\xk (\givenPrm), \givenPrm) \ud$ respectively converge to $\derivx \fixMap (\xmin, \givenPrm)$ and $\derivu \fixMap (\xmin, \givenPrm) \ud$ for any $\ud\in\spPrm$. Furthermore, the rate of convergence of the two sequences is linear and is the same as that of $\xk$, that is, $q_{\x}$. Therefore, the proof follows by simply invoking the second part of Theorem~\ref{thm:LGFPI:conv}.
\end{proof}
\begin{remark}
    Assumption~\ref{ass:linear:convergence} is not practical and therefore we rarely see any application of the above result in practice. For instance, for gradient descent with line search, it requires linear convergence of the step size sequence which does not hold in general.
\end{remark}

\section{Proofs of Section~\ref{sec:FPI:k}}

\appSubSect{Lemma}{lem:IFT:FPE:k}
\vspace{2ex}
\begin{proof}
    From Assumption~\ref{ass:basic:modified}\ref{itm:basic:modified:spectral:norm}, given $\eps \in (0, 1-\rho)$, there exists $K\in\N$ such that $\rho (\derivx \fixMap[k] (\xmin, \givenPrm)) \leq \norm{\derivx \fixMap[k] (\xmin, \givenPrm)} < \rho + \eps$ for all $k\geq K$, where $\rho \coloneqq \limsup_{k\to\infty} \norm{\derivx \fixMap[k] (\xmin, \givenPrm)}$. The existence of the neighbourhood $\neighbourhoodPRM$ and the $C^1$-smooth mapping $\map{\argmap}{\neighbourhoodPRM}{\spVar}$ is guaranteed by Theorem~\ref{thm:IFT:FPE} applied to $\fixMap[n]$ for some $n\geq K$, thanks to Assumption~\ref{ass:basic:modified} (\ref{itm:basic:modified:C1}--\ref{itm:basic:modified:fixed:point}). In particular, for all $\u\in \neighbourhoodPRM$, $\argmap (\u) = \fixMap[n] (\argmap (\u), \u)$ and $D\argmap (\u)$ is given by \eqref{eq:FPE:IFT} with $\fixMap$ replaced by $\fixMap[n]$. From Assumption~\ref{ass:basic:modified}\ref{itm:basic:modified:fixed:point}, for any $k\geq K$ and $\u\in \neighbourhoodPRM$, we have $\argmap (\u) = \fixMap[n] (\argmap (\u), \u) = \fixMap[k] (\argmap (\u), \u)$ and with $\neighbourhoodPRM$ possibly reduced, $\rho (\derivx \fixMap[k] (\argmap (\u), \u)) < 1$. Therefore, by using the Chain rule and Lemma~\ref{lem:spec_rad:inv_op}, we obtain the expression in \eqref{eq:FPE:IFT:k}.
\end{proof}

\appSubSect{Theorem}{thm:D:FPI:modified:conv}
\vspace{2ex}
\begin{proof}
    From Lemma~\ref{lem:IFT:FPE:k} and Remark~\ref{rem:IFT:FPE:k}, there exists $K\in\N$ such that for any $k\geq K$, the fixed-point mapping $\argmap_k$ of $\fixMap[k] (\cdot, \u)$ is $C^1$-smooth near $\givenPrm$ and $\deriv \argmap_k (\givenPrm) = X_*$ (see Assumption~\ref{ass:basic:modified}\ref{itm:basic:modified:fixed:point}). We denote $\xk\coloneqq \xk (\givenPrm)$ for simplicity and define
    \begin{equation} \label{eq:sequences:D:FPI:k:conv}
        \begin{aligned}
            \ek &\coloneqq \left( \deriv \xk (\givenPrm) - X_* \right) \ud \\
            B_k &\coloneqq \derivx \fixMap[k] (\xmin, \givenPrm) \\
            C_k &\coloneqq \derivx \fixMap[k] (\xk, \givenPrm) - \derivx \fixMap[k] (\xmin, \givenPrm) \\
            \dk &\coloneqq \left(\deriv \fixMap[k] (\xk, \givenPrm) - \deriv \fixMap[k] (\xmin, \givenPrm)\right) \begin{bmatrix}
                X_*\ud \\ \ud
            \end{bmatrix} \,,
        \end{aligned}
    \end{equation}
    to obtain
    \begin{equation} \label{eq:ekp:D:FPI:k:conv}
        \begin{aligned}
            \ekp &= \Big( \derivx \fixMap[k] (\xk, \givenPrm) D \xk (\givenPrm) \ud + \derivu \fixMap[k] (\xk, \givenPrm) \ud \Big) -\\
            &\Big( \derivx \fixMap[k] (\xmin, \givenPrm) X_*\ud + \derivu \fixMap[k] (\xmin, \givenPrm) \ud \Big) \\
            &= B_k \ek + C_k \zk + \dk \,.
        \end{aligned}
    \end{equation}
    From Assumption~\ref{ass:convergence:modified} and the definitions of $C_k$ and $\yk$, we note that $C_k \to 0$ and $\dk \to 0$. Therefore, from Theorem~\ref{thm:LGFPI:modified:conv}, $\seq[k\in\N_0]{\deriv \xk (\givenPrm) \ud}$ converges to $X_* \ud$ for any $\ud\in\spPrm$.
\end{proof}

\appSubSect{Theorem}{thm:D:FPI:modified:conv:linear}
\vspace{2ex}
\begin{proof}
    From Assumption~\ref{ass:linear:convergence:modified}, we have $\deriv \fixMap[k] (\xk (\givenPrm), \givenPrm) - \deriv \fixMap[k] (\xmin, \givenPrm) \to 0$ and the rate of convergence of the two sequences is $q_{\x}$. For any $\ud\in\spPrm$, by using the definitions in \eqref{eq:sequences:D:FPI:k:conv}, we note that both $C_k$ and $\yk$ converge linearly with rate $q_{\x}$. Thus, invoking the second part of Theorem~\ref{thm:LGFPI:modified:conv}, we obtain the asymptotic and the non-asymptotic convergence rate of $\seq[k\in\N_0]{\deriv \xk (\givenPrm) \ud}$.
\end{proof}

\section{Proofs of Section~\ref{sec:FBS}}
\appSubSect{Theorem}{thm:PGD:Derv}
The proof of Theorem~\ref{thm:PGD:Derv} relies on the following preliminary result from set-valued analysis.
\begin{lemma} \label{lem:relint:Convex:Set:Sequence}
    Given a sequence of non-empty, closed, convex sets $C_k \subset \spVar$ which converges to $C \subset \spVar$. Let $\xk \in C_k$ be a sequence with limit $\x\in\ri C$. Then there exists $K\in\N$ such that $\xk \in \ri C_k$ for all $k\geq K$.
\end{lemma}
\begin{proof}
    Because $\xk \to \x$ and $\x \in \ri C$, there exists a compact set $V$ and $K\in\N$ such that $\xk, \x \in \sint V \neq \emptyset$ for all $k\geq K$ and $V \cap \aff C \subset C$. Moreover, the convergence of $C_k$ implies 
    $\aff C_k \to \aff C$ which implies $V\cap\aff C_k \to V\cap\aff C$ \cite[Lemma~1.4]{Mos69}. Once we show that $V\cap\aff C_k \subset C_k$ eventually, we are done. Assume for contradiction, that there exists a subsequence $\seq[i\in\N]{C_{k_i}}$ such that for all $i\in\N$, $\y\iter{i} \in V\cap\aff C_{k_i}$ and $\y\iter{i} \notin C_{k_i}$. The compactness of $V$ and convergence of $V\cap\aff C_{k_i}$ implies the existence of $\y \in V\cap\aff C$ such that $\y\iter{i} \to \y$, possibly through a subsequence. We define the bounded sequence $\b\iter{i}\coloneqq \bm{a}\iter{i} / \norm{\bm{a}\iter{i}}$ where $\bm{a}\iter{i} \coloneqq \y\iter{i} - \proj{C_{k_i}} (\y\iter{i})$ which lies in $\spPar C_{k_i}$ and has limit $0\neq\b\in \spPar C$, again, possibly through a subsequence. For all $i\in\N$ and for all $\w\in C_{k_i}$,
    \begin{equation*}
        \begin{aligned}
            \scal{\b\iter{i}}{\w - \y\iter{i}} &= -\frac{1}{\norm{\bm{a}\iter{i}}} \left( \scal{\bm{a}\iter{i}}{\y\iter{i} - \proj{C_{k_i}} (\y\iter{i})} + \scal{\bm{a}\iter{i}}{\proj{C_{k_i}} (\y\iter{i}) - \w} \right) \\
            &\leq -\frac{1}{2} \norm{\bm{a}\iter{i}} \leq 0 \,,
        \end{aligned}
    \end{equation*}
    where the inequality follows because $\bm{a}\iter{i} \in \ncone{C_{k_i}} (\proj{C_{k_i}} (\y\iter{i}))$. The above inequality leads to $\scal{\b}{\w - \y} \leq 0$ for all $\w \in C$. In other words, we obtain $0\neq\b\in \ncone{C} (\y) \cap \spPar C$ which is a contradiction because $\y\in\ri C$.
\end{proof}
\begin{proof}[Proof of Theorem \ref{thm:PGD:Derv}]
    \begin{enumerate}[label=(\roman*)]
        \item From the convergence of $\xk \coloneqq \xk (\givenPrm)$, Lemma~\ref{lem:APG} ensures that for $\delta>0$ small enough, there exists $K\in\N$, such that $\xk\in\manif \cap B_{\delta} (\xmin)$ for all $k\geq K$. We define the maps $\map{G}{\spVar\times\setPrm\times(0, 2/L)}{\spVar}$ by $G(\x, \u, \alpha) = \x - \alpha \grad[\x] f(\x, \u)$ and $\map{H}{\spVar\times\spVar\times\setPrm\times(0, 2/L)}{\eR}$ by
    \begin{equation*}
        H(\y, \x, \u, \alpha) \coloneqq \alpha g(\y, \u) + \frac{1}{2} \norm{\y - \x + \alpha \grad[\x] f(\x, \u)}^2 \,,
    \end{equation*}
    and note that $\xkp = \pgd[k] (\xk, \givenPrm) = \argmin_{\y} H(\y, \xk, \givenPrm, \alpha_k)$, which is equivalent to $0\in\partial_{\y} H (\xkp, \xk, \givenPrm, \alpha_k)$ or $\bmuk \in \partial_{\x} F(\xkp, \givenPrm)$ from Fermat's rule, where
    \begin{equation} \label{eq:APG:Fermat}
        \bmuk \coloneqq \frac{1}{\alpha_k} \Big( G (\xk, \givenPrm, \alpha_k) - G (\xkp, \givenPrm, \alpha_k)\Big)\,.
    \end{equation}
    Notice that $\bmuk\to0$ because $G(\cdot, \givenPrm, \alpha)$ is non-expansive when $\alpha\in(0, 2/L)$ and we have $\norm{\bmuk} \leq \frac{1}{\sslow} \norm{\xk - \xkp}$. Therefore, from \eqref{eq:ND} and Lemma~\ref{lem:relint:Convex:Set:Sequence}, we have $\bmuk\in\ri \partial_{\x} F(\xkp, \givenPrm)$ or $0\in\ri\partial_{\y} H (\xkp, \xk, \givenPrm, \alpha_k)$, for all $k\in\N$ large enough. Since $H$ is partly smooth relative to $\manif\times\spVar\times\setPrm\times(0, 2/L)$ and $H(\cdot, \x, \u, \alpha)$ is strongly convex for all $(\x, \u, \alpha) \in \spVar\times\setPrm\times(0, 2/L)$, and non-degeneracy condition holds, the $C^1$-smoothness of the update map near $(\xk, \givenPrm, \alpha_k)$ follows by invoking Theorem~\ref{thm:IFT}. The $C^1$-smoothness of $\pgd$ near $(\xmin, \givenPrm, \alpha)$ under given assumptions was shown in \cite[Corollary~32]{MO24} for any $\alpha\in[\sslow, \ssup]$.
    \item We define
    \begin{equation*}
        \begin{gathered}
            Q_k \coloneqq \Hess[\manif] H (\xkp, \xk, \givenPrm, \alpha_k) \,, \quad \tilde{Q}_k \coloneqq \Hess[\manif] H (\xmin, \xmin, \givenPrm, \alpha_k) \\
            P_{\omega, k} \coloneqq \deriv[\omega] \grad[\manif] H (\xkp, \xk, \givenPrm, \alpha_k) \,, \quad \tilde{P}_{\omega, k} \coloneqq \deriv[\omega] \grad[\manif] H (\xmin, \xmin, \givenPrm, \alpha_k) \\
            \projTan_k \coloneqq \projTan (\xk) \,, \quad \projTan_* \coloneqq \projTan (\xmin) \,, \quad \projNor_k \coloneqq \projNor (\xk) \,, \quad \projNor_* \coloneqq \projNor (\xmin) \,,
        \end{gathered}
    \end{equation*}
    where $\omega \in \set{\x, \u, \alpha}$ and evaluate $Q_k$ to obtain
    \begin{equation*}
        \begin{aligned}
            Q_k &= \alpha_k\Hess[\manif] g (\xkp, \givenPrm) + \projTan_{k+1} + \wein[\xkp]{\cdot}{\projNor_{k+1} (\xkp - G(\xk, \givenPrm, \alpha_k)} \\
            &= \alpha_k \Hess[\manif] g (\xmin, \givenPrm) + \alpha_k \wein[\xkp]{\cdot}{\projNor_{k+1} \grad[\x] f (\xkp, \givenPrm)} + \projTan_{k+1} + \\
            &\wein[\xkp]{\cdot}{\projNor_{k+1} (G(\xkp, \givenPrm, \alpha_k) - G(\xk, \givenPrm, \alpha_k))} \\
            &= \alpha_k \Hess[\manif] F (\xmin, \givenPrm) - \projTan_{k+1} \Hess[\x] f (\xkp, \givenPrm) \projTan_{k+1} + \projTan_{k+1} + \\
            &\wein[\xkp]{\cdot}{\projNor_{k+1} (G(\xkp, \givenPrm, \alpha_k) - G(\xk, \givenPrm, \alpha_k))} \,.
        \end{aligned}
    \end{equation*}
    Similarly, $\tilde{Q}_k$ is given by
    \begin{equation*}
        \tilde{Q}_k = \alpha_k \Hess[\manif] F (\xmin, \givenPrm) - \alpha_k \projTan_* \Hess[\x] f (\xmin, \givenPrm) \projTan_* + \projTan_* \,.
    \end{equation*}
    From \cite[Lemma~4.3]{LFP17}, $\alpha_k \Hess[\manif] F (\xmin, \givenPrm) - \alpha_k \projTan_* \Hess[\x] f (\xmin, \givenPrm) \projTan_*$ is positive semi-definite and the eigenvalues of $\tilde{Q}_k^{\dagger}$ lie in $(0, 1]$. Moreover, the non-expansiveness of $G(\cdot, \givenPrm, \alpha_k)$ and the continuity of $(\x, \v) \mapsto \wein[\x]{\cdot}{\v}$ due to the $C^2$-smoothness of $\manif$ implies that $\wein[\xkp]{\cdot}{\projNor_{k+1} (G(\xkp, \givenPrm, \alpha_k) - G(\xk, \givenPrm, \alpha_k))}\to0$. This entails that $Q_k - \tilde{Q}_k \to 0$ and the eigenvalues of $Q_k^{\dagger}$ are also eventually bounded. Hence, $Q_k^{\dagger} - \tilde{Q}_k^{\dagger}$, which can be rewritten as
    \begin{equation*}
        \begin{aligned}
            Q_k^{\dagger} - \tilde{Q}_k^{\dagger} &= Q_k^{\dagger} - Q_k^{\dagger} \projTan_* + Q_k^{\dagger} \projTan_* - \projTan_k \tilde{Q}_k^{\dagger} + \projTan_k \tilde{Q}_k^{\dagger} - \tilde{Q}_k^{\dagger} \\
            &= Q_k^{\dagger} \projTan_k - Q_k^{\dagger} \projTan_* + Q_k^{\dagger} \tilde{Q}_k \tilde{Q}_k^{\dagger} - Q_k^{\dagger} Q_k Q_*^{\dagger} + \projTan_k \tilde{Q}_k^{\dagger} - \tilde{Q}_k^{\dagger} \projTan_* \\
            &= Q_k^{\dagger} \left(\projTan_k - \projTan_*\right) - Q_k^{\dagger}\left(Q_k - \tilde{Q}_k\right)\tilde{Q}_k^{\dagger} + \left(\projTan_k - \projTan_*\right) \tilde{Q}_k^{\dagger} \,,
        \end{aligned}
    \end{equation*}
    also converges to $0$. Similarly $P_{\omega, k} - \tilde{P}_{\omega, k} \to 0$ for all $\omega\in\set{\x, \u, \alpha}$ because
    \begin{equation*}
        \begin{aligned}
            P_{\x, k} - \tilde{P}_{\x, k} &= \projTan_{k+1} \left( \opid - \alpha_k \Hess[\x] f (\xk, \givenPrm) \right) - \projTan_{*} \left(\opid - \alpha_k \Hess[\x] f (\xmin, \givenPrm) \right) \\
            &= \left( \projTan_{k+1} - \projTan_{*} \right) - \alpha_k \left( \projTan_{k+1} \Hess[\x] f (\xk, \givenPrm) - \projTan_{*} \Hess[\x] f (\xmin, \givenPrm) \right) \\
            P_{\u, k} - \tilde{P}_{\u, k} &= \alpha_k \left( D_{\u} \grad[\manif] g (\xkp, \givenPrm) + D_{\u} \projTan_{k+1} \grad[\x] f (\xk, \givenPrm) - D_{\u} \grad[\manif] F (\xmin, \givenPrm) \right) \\
            P_{\alpha, k} - \tilde{P}_{\alpha, k} &= \grad[\manif] g (\xkp, \givenPrm) + \projTan_{k+1} \grad[\x] f (\xk, \givenPrm) - \grad[\manif] F (\xmin, \givenPrm) \\
        \end{aligned}
    \end{equation*}
    This concludes the proof because from \eqref{eq:IFT}, $\deriv \pgd (\x, \u)$ is given by,
    \begin{equation} \label{eq:PGD:Derv}
        \deriv \pgd (\x, \u) = -\Hess[\manif] H (\pgd (\x, \u), \x, \u, \alpha)^{\dagger} \deriv[(\x, \u, \alpha)] \grad[\manif] H (\pgd (\x, \u), \x, \u, \alpha) \,.
    \end{equation}
    \item The expressions for $Q_{k} - \tilde{Q}_{k}$ and $P_{\omega, k} - \tilde{P}_{\omega, k}$ for $\omega\in\set{\x, \u, \alpha}$ clearly indicate that these sequences converge linearly under the given assumptions.
    \end{enumerate}
\end{proof}
\appSubSect{Theorem}{thm:PGD:AD}
\vspace{2ex}
\begin{proof}
    We again set $\xk \coloneqq \xk (\givenPrm)$ for simplicity. Thanks to Theorem~\ref{thm:PGD:Derv}\ref{itm:PGD:Derv:Itrs}, $\pgd[k]$ are $C^1$-smooth near $(\xk, \givenPrm, \alpha_k)$ for all $k\in\N$ provided that $\xz$ is sufficiently close to $\xmin$. Differentiation of the fixed-point iteration $\xkp \coloneqq \pgd[k] (\xk, \givenPrm)$ with respect to $\u$ yields
    \begin{equation} \label{eq:PGD:AD}
        \deriv \xkp (\givenPrm) = \derivx \pgd[k] (\xk, \givenPrm) \projTan (\xk) \deriv \xk (\givenPrm) + \derivu \pgd[k] (\xk, \givenPrm) \,,
    \end{equation}
    and Theorem~\ref{thm:PGD:IFT} asserts that under the given assumptions, for all $k\in\N$,
    \begin{equation} \label{eq:PGD:FPE:Derv}
        \deriv \argmap (\givenPrm) = \derivx \pgd[k] (\xmin, \givenPrm) \projTan (\xmin) \deriv \argmap (\givenPrm) + \derivu \pgd[k] (\xmin, \givenPrm) \,.
    \end{equation}
    Just like the arguments made in the proof of Theorem~\ref{thm:D:FPI:conv}, for any $\ud\in\spPrm$, if we define
    \begin{equation*}
        \begin{gathered}
            \ek\coloneqq \deriv \xk (\givenPrm) \ud - \deriv \argmap (\givenPrm) \ud \\
            B_k\coloneqq \derivx \pgd[k] (\xmin, \givenPrm) \projTan (\xmin) \\
            C_k\coloneqq \derivx \pgd[k] (\xk, \givenPrm) \projTan (\xk) - \derivx \pgd[k] (\xmin, \givenPrm) \projTan (\xmin) \\
            \dk \coloneqq (\deriv \fixMap[k] (\xk, \givenPrm) - \deriv \fixMap[k] (\xmin, \givenPrm)) (D \argmap (\givenPrm) \ud, \ud) \,,
        \end{gathered}
    \end{equation*}
    and subtract \eqref{eq:PGD:FPE:Derv} from \eqref{eq:PGD:AD}, we obtain the recursion $\ekp = B_k\ek + C_k\ek + \dk$. From Theorem~\ref{thm:PGD:IFT}, the continuity of $\rho$ and $\deriv \pgd$, $\sup_{k\in\N} \norm{B_k} < 1$ for all $k\in\N$ and from Theorem~\ref{thm:PGD:Derv}\ref{itm:PGD:Derv:Conv}, $C_k \to 0$ and $\dk \to 0$. Under the additional assumptions, $C_k$ and $\dk$ converge linearly with rate $q_{\x}$ due to Theorem~\ref{thm:PGD:Derv}\ref{itm:PGD:Derv:Itrs}. Therefore, the (linear) convergence of the derivative sequence $\seq[k\in\N]{\deriv \xk (\givenPrm)}$ follows from Theorem~\ref{thm:LGFPI:modified:conv}.
\end{proof}
\appSubSect{Theorem}{thm:APG:Derv}
\vspace{2ex}
\begin{proof}
    \begin{enumerate}[label=(\roman*)]
        \item This part follows from \cite[Corollary~33]{MO24} and the convergence of $\xk (\givenPrm)$ and $\alpha_k$.
        \item We set $\z\coloneqq(\x_1, \x_2)$, and $\y\coloneqq\x_1 + \beta (\x_1 - \x_2)$ and write the expression for $\deriv \apg (\z, \u)$ for $(\z, \u, \alpha, \beta) \in \neighbourhoodAPG[*]\times [0, 1]$:
        \begin{equation*}
            \begin{aligned}
                D_{\z} \apg (\z, \u) &= \begin{bmatrix}
                    (1+\beta) D_{\x} \pgd (\y, \u) & -\beta D_{\x} \pgd (\y, \u) \\
                    \opid & 0
                \end{bmatrix} \\
                D_{\u} \apg (\z, \u) &= \begin{bmatrix}
                    D_{\u} \pgd (\y, \u) \\
                    0
                \end{bmatrix} \\
                D_{\alpha} \apg (\z, \u) &= \begin{bmatrix}
                    D_{\alpha} \pgd (\y, \u) \\
                    0
                \end{bmatrix}\,,
            \end{aligned}
        \end{equation*}
        as provided in \cite[Corollary~33]{MO24}, where $\deriv \pgd (\y, \u)$ is given in \eqref{eq:PGD:Derv}. It is easy to see that the mapping $(\z, \u, \alpha, \beta) \mapsto \deriv \apg (\z, \u) $ is continuous on $\neighbourhoodAPG[*]\times[0, 1]$. Therefore, the sequences $\seq[k\geq K]{\deriv \apg[k] (\xk (\givenPrm), \xkm (\givenPrm), \givenPrm)}$ and $\seq[k\geq K]{\deriv \apg[k] (\xmin, \xmin, \givenPrm)}$ both converge to $\deriv \apg[*] (\xmin, \xmin, \givenPrm)$ and their difference converges to $0$.
        \item Because $\beta_k \in [0, 1]$ and under the additional assumptions, $\deriv \pgd[k] (\yk, \givenPrm) - \deriv \pgd[k] (\xmin, \givenPrm) = \O (\yk - \xmin) $ as established in Theorem~\ref{thm:PGD:Derv}\ref{itm:PGD:Derv:Conv:Linear}, where $\yk \coloneqq \xk (\givenPrm) + \beta_k (\xk (\givenPrm) - \xkm (\givenPrm))$, the result follows directly from the expressions of $\deriv \apg$.
    \end{enumerate}
\end{proof}
\appSubSect{Theorem}{thm:APG:AD}
\vspace{2ex}
\begin{proof}
    From Theorem~\ref{thm:APG:Derv} the mapping $(\x_1, \x_2, \u, \alpha)\mapsto\apg[k] (\x_1, \x_2, \u, \alpha)$ is $C^1$-smooth on $\neighbourhoodAPG[*]$. Therefore, when $\xz$ is close enough to $\xmin$, we have $(\xk, \xkm, \givenPrm, \alpha_k) \in \neighbourhoodAPG[*]$ where $\xk \coloneqq \xk (\givenPrm)$ and we obtain the following recursion for the derivative iterates
    \begin{equation*}
        \begin{bmatrix}
            \deriv \xkp (\givenPrm) \\
            \deriv \xk (\givenPrm)
        \end{bmatrix} = \derivz \apg[k] (\xk, \givenPrm) \projTan (\xk, \xkm) \begin{bmatrix}
            \deriv \xk (\givenPrm) \\
            \deriv \xkm (\givenPrm)
        \end{bmatrix} + \derivu \apg[k] (\xk, \givenPrm)
    \end{equation*}
    From Theorem~\ref{thm:APG:IFT}, we have
    \begin{equation*}
        \begin{bmatrix}
            \deriv \argmap (\givenPrm) \\
            \deriv \argmap (\givenPrm)
        \end{bmatrix} = \derivz \apg (\xmin, \givenPrm) \projTan (\xmin, \xmin) \begin{bmatrix}
            \deriv \argmap (\givenPrm) \\
            \deriv \argmap (\givenPrm)
        \end{bmatrix} + \derivu \apg (\xmin, \givenPrm) \,,
    \end{equation*}
    for $\beta = \beta_*$ and $\beta = \beta_k$ for all $k\in\N$. Therefore, for any $\ud\in\spPrm$, we define 
    \begin{equation*}
        \begin{aligned}
            \ek &\coloneqq (\deriv\xk (\givenPrm) - \deriv \argmap (\givenPrm), \deriv\xkm (\givenPrm) - \deriv \argmap (\givenPrm))\ud \\
            B_k &\coloneqq \derivx \apg[k] (\xmin, \xmin, \givenPrm) \projTan (\xmin, \xmin) \,,\quad B_* \coloneqq \derivx \apg[*] (\xmin, \xmin, \givenPrm) \projTan (\xmin, \xmin) \\
            C_k &\coloneqq \derivx \apg[k] (\xk, \xkm, \givenPrm) \projTan (\xk, \xkm) - \derivx \apg[k] (\xmin, \xmin, \givenPrm) \projTan (\xmin, \xmin) \\
            \dk &\coloneqq (\deriv \apg[k] (\xk, \xkm, \givenPrm) - \deriv \apg[k] (\xmin, \xmin, \givenPrm))(\deriv\argmap (\givenPrm) \ud, \deriv\argmap (\givenPrm)\ud, \ud)
        \end{aligned}
    \end{equation*}
    and obtain the recursion $\ekp = B_k\ek + C_k\ek + \dk$. Notice that $\rho (B_*) < 1$ from Theorem~\ref{thm:APG:IFT} and $B_k \to B_*$ from Theorem~\ref{thm:APG:Derv}\ref{itm:APG:Derv:Itrs}. Moreover $C_k \to 0$, and $\dk \to 0$ owing to Theorem~\ref{thm:APG:Derv}\ref{itm:APG:Derv:Conv} and under the additional assumptions, the convergence of $C_k$ and $\dk$ is linear with rate $q_{\x}$ thanks to Theorem~\ref{thm:APG:Derv}\ref{itm:APG:Derv:Conv:Linear}. Therefore,
    from Corollary~\ref{cor:LGFPI:modified:conv}, the sequence $\deriv \xk (\givenPrm) \ud$ converges (linearly) to $\deriv \argmap (\givenPrm) \ud$ for all $\ud \in \spPrm$.
\end{proof}

\section{Experimental Details} \label{sec:Exp:supp}
In this section, we fill out some missing details from Section~\ref{sec:exp}. For \eqref{prob:logistic}, we use MADELON dataset \cite{DG17} where $A \in \R^{M\times N}$ and $\b \in \set{-1, +1}^{M}$ where $M=2000$ and $N=501$ without any feature transformation apart from batch normalization. For $50$ set of experiments, we sample $\lambda \sim \mathcal {N} (0, 10^{-3})$ and perturb each element of $A$ with a Gaussian noise with standard deviation $10^{-3}$. The problem is $\mu$-strongly convex and has an $L$-Lipschitz continuous gradient where $\mu = \lambda$ and $L = \rho(A^{T}A) + \lambda$. For \eqref{prob:lasso}, the dataset $(A, \b)\in\R^{M\times N} \times \R^M$ with $70\leq M\leq 90$ and $N=200$ is artificially generated. In particular, for each of the $50$ problems, we generate $A$ and a sparse vector $\x^{\prime} \in \R^N$ with only $50$ non-zero elements. We sample each element of $A$ and each non-zero element of $\x^{\prime}$ from $U (0, 1)$ and $\mathcal{N} (0, 1)$ respectively. We then generate $\b$ by computing $A\x + \bm\eps$ where $\bm\eps \sim \mathcal N (0, 10^{-3}\opid)$. The problem has an $L$-Lipschitz continuous gradient with $L = \rho (A^{T}A)$ but is not even strictly convex. However, for each choice of $\lambda \sim U (0, 10)$, $A$, and $\b$, Assumptions~\ref{ass:RPD} and \ref{ass:ND} were satisfied for each experiment.

We obtain a very good estimate of $\xmin$ for \eqref{prob:logistic} by using Newton's method with backtracking line search and for \eqref{prob:lasso} by running APG for sufficiently large number of iterations. The derivative of the solution $D\argmap (\givenPrm)$ is computed by solving the linear system \eqref{eq:FPE:IFT:intro} for logistic regression and solving the reduced system \eqref{eq:IFT} after identifying the support for lasso regression \cite{BKM+22}. For each problem and for each experiment, we run PGD with four different choices of step size, namely, (i) $\alpha_k = 2/(L+m)$ for \eqref{prob:logistic} and $\alpha_k = 1/L$ for \eqref{prob:lasso}, (ii) $\alpha_k \sim U (0, \frac{2}{3L})$, (iii) $\alpha_k \sim U (\frac{2}{3L}, \frac{4}{3L})$ , and (iv) $\alpha_k \sim U (\frac{4}{3L}, \frac{2}{L})$, for each $k\in\N$. We also run APG with $\alpha_k = 1/L$ and $\beta_k = (k-1)/(k+5)$. Before starting each algorithm, we obtain $\xz \in B_{10^{-2}} (\xmin)$ by partially solving each problem through APG. For computational reasons, instead of generating the sequence $\deriv \xk (\u)$ where $\u = (A, \lambda)$ and $\u = (A, \b, \lambda)$ respectively for the two problems, we compute the directional derivative $\deriv \xk (\u) \ud$. The vector $\ud$ belongs to the same space as $\u$ and is fixed for the $5$ different algorithms being compared for each experiment and problem.

\section*{Acknowledgments}
Sheheryar Mehmood and Peter Ochs are supported by the German Research Foundation
(DFG Grant OC 150/4-1).


{\small
\bibliographystyle{ieee}
\bibliography{main}
}

\end{document}